\documentclass{amsart}

\usepackage{enumitem}

\usepackage{amsmath,amsthm, amscd, amssymb, amsfonts, mathrsfs}

\usepackage{float}
\usepackage[T1]{fontenc}

\usepackage{color}
\usepackage[dvipsnames]{xcolor}
\usepackage{mathtools}

\usepackage{cite}

\usepackage[plainpages=false,pdfpagelabels,
hypertexnames=false]{hyperref}

\usepackage[all]{xy}
\usepackage{multicol}
\usepackage{tikz}
\usetikzlibrary{patterns}
\usetikzlibrary{arrows}
\usetikzlibrary{positioning}
\usetikzlibrary{decorations.pathmorphing}
\usetikzlibrary{cd}

\mathchardef\mhyphen="2D

\allowdisplaybreaks

\newcommand{\grk}{\operatorname{grk}}
\newcommand{\Kar}{\operatorname{Kar}}

\newcommand{\cC}{\mathscr{C}}
\newcommand{\cT}{\mathcal{T}}
\newcommand{\cF}{\mathcal{F}}
\newcommand{\cG}{\mathcal{G}}
\newcommand{\cH}{\mathcal{H}}
\newcommand{\cP}{\mathcal{P}}
\newcommand{\cL}{\mathcal{L}}

\newcommand{\ku}{\Bbbk}

\newcommand\fP{\mathsf{P}}

\newcommand{\Z}{{\mathbb Z}}

\newcommand{\V}{{\mathbb V}}

\newcommand{\End}{\operatorname{End}}

\newcommand{\Exp}{\operatorname{Exp}}

\newcommand\Hom{\operatorname{Hom}}

\newcommand\id{\operatorname{id}}

\numberwithin{equation}{section}
\newtheorem{lema}{Lemma}[section]
\newtheorem{theorem}[lema]{Theorem}
\newtheorem{cor}[lema]{Corollary}

\newtheorem{prop}[lema]{Proposition}

\newtheorem{question-app}{Question}

\theoremstyle{definition}

\theoremstyle{remark}
\newtheorem{example}[lema]{Example}

\newtheorem{rmk}[lema]{Remark}

\newcommand{\uv}{{\underline{v}}}
\newcommand{\uw}{{\underline{w}}}
\newcommand{\ux}{{\underline{x}}}
\newcommand{\uy}{{\underline{y}}}

\newcommand{\uH}{\underline{H}}

\newcommand{\uPhi}{\underline{\Phi}}

\newcommand{\HHom}{\mathbb{H}\hspace{-1.5pt}\operatorname{om}}
\newcommand{\EEnd}{\mathbb{E}\hspace{-1.5pt}\operatorname{nd}}

\newcommand{\ustar}{\,\underline{\star}\,}
\newcommand{\wstar}{\,\widehat{\star}\,}

\newcommand{\oDelta}{\overline{\Delta}}
\newcommand{\onabla}{\overline{\nabla}}
\newcommand{\uDelta}{\underline{\Delta}}
\newcommand{\unabla}{\underline{\nabla}}

\newcommand{\wnabla}{\widetilde{\nabla}}
\newcommand{\uB}{\underline{B}}

\newcommand{\dDelta}{\hspace{1pt} {\mbox{\scriptsize$\Delta$}}\kern-8pt\Delta}

\newcommand{\dnabla}{\rotatebox[origin=c]{180}{${\mbox{\scriptsize$\Delta$}}\kern-8pt\Delta$}}

\newcommand{\fh}{\mathfrak{h}}

\newcommand{\rT}{\mathscr{T}}

\newcommand{\rL}{\mathscr{L}}

\newcommand{\uT}{\underline{\cT}}
\newcommand{\oT}{\overline{\cT}}
\newcommand{\wT}{\widetilde{\cT}}
\newcommand{\wP}{\widetilde{\cP}}

\newcommand{\wb}{\widehat{b}}
\newcommand{\ws}{\widehat{s}}
\newcommand{\wt}{\widehat{t}}

\newcommand{\wepsilon}{\widehat{\epsilon}}
\newcommand{\weta}{\widehat{\eta}}

\newcommand\BE{\mathsf{BE}}
\newcommand\RE{\mathsf{RE}}
\newcommand\LE{\mathsf{LE}}
\newcommand\LM{\mathsf{LM}}
\newcommand\FM{\mathsf{FM}}
\newcommand\Loc{\mathsf{Loc}}
\newcommand\loc{\mathrm{loc}}
\newcommand\Tilt{\mathsf{Tilt}}

\newcommand{\For}{\mathsf{For}}

\newcommand{\bk}{\Bbbk}

\newcommand{\Mod}{\mathrm{Mod}}

\newcommand{\Kb}{K^{\mathrm{b}}}
\newcommand{\Db}{D^{\mathrm{b}}}

\newcommand{\DiagBS}{\mathscr{D}_{\mathrm{BS}}}
\newcommand{\oDiagBS}{\overline{\mathscr{D}}_{\mathrm{BS}}}
\newcommand{\uDiagBS}{\underline{\mathscr{D}}_{\mathrm{BS}}}
\newcommand{\Diag}{\mathscr{D}}
\newcommand{\oDiag}{\overline{\mathscr{D}}}
\newcommand{\uDiag}{\underline{\mathscr{D}}}

\newcommand{\oB}{\overline{B}}
\newcommand{\oD}{\overline{\Delta}}

\newcommand{\Abe}{\mathscr{A}}
\newcommand{\AbeBS}{\mathscr{A}_{\mathrm{BS}}}
\newcommand{\TBS}{\mathscr{T}_{\mathrm{BS}}}
\newcommand{\simto}{\xrightarrow{\sim}}

\makeatletter
\def\Tenint{\@ifnextchar_{\@Tenintsub}{\@Tenintnosub}}
\def\@Tenintsub_#1{\mathchoice{\mathbin{\underline{\mathop{\otimes}}}_{#1}}%
  {\underline{\otimes}_{#1}}{\underline{\otimes}_{#1}}{\underline{\otimes}^L_{#1}}}
\def\@Tenintnosub{\mathbin{\underline{\mathop{\otimes}}}}
\makeatother

\makeatletter
\def\lotimes{\@ifnextchar_{\@lotimessub}{\@lotimesnosub}}
\def\@lotimessub_#1{\mathchoice{\mathbin{\mathop{\otimes}^L}_{#1}}%
  {\otimes^L_{#1}}{\otimes^L_{#1}}{\otimes^L_{#1}}}
\def\@lotimesnosub{\mathbin{\mathop{\otimes}^L}}
\makeatother

\title{Koszul duality for Coxeter groups}

\author{Simon Riche}
\address{Universit\'e Clermont Auvergne, CNRS, LMBP, F-63000 Clermont-Ferrand, Fran\-ce.
}
\email{simon.riche@uca.fr}

\author{Cristian Vay}
\address{Universidad Nacional de C\'ordoba, Facultad de Matem\'atica, Astronom\'ia, F\'isica y Computaci\'on, CIEM -- CONICET, C\'ordoba, Argentina.}
\email{cristian.vay@unc.edu.ar}

\thanks{This project has received funding from the European Research Council (ERC) under the European Union's Horizon 2020 research and innovation programme (S.R., grant agreement No.~101002592). C. V. is partially supported by  CONICET PIP 11220200102916CO, Foncyt PICT 2019-3660 and Secyt (UNC)}

\begin{document}

\begin{abstract}
We construct a ``Koszul duality'' equivalence relating the (diagrammatic) Hecke category attached to a Coxeter system and a given realization to the Hecke category attached to the same Coxeter system and the dual realization. This extends a construction of Be{\u\i}linson--Ginzburg--Soergel~\cite{bgs} and Bezrukavni\-kov--Yun~\cite{by} in a geometric context, and of the first author with Achar, Makisumi and Williamson~\cite{amrw2}. As an application, we show that the combinatorics of the ``tilting perverse sheaves'' considered in~\cite{arv} is encoded in the combinatorics of the canonical basis of the Hecke algebra of $(W,S)$ attached to the dual realization.
\end{abstract}

\maketitle

\section{Introduction}

\subsection{Koszul duality for general Coxeter groups}

The utility of Koszul duality in Representation Theory has been first emphasized by Be{\u\i}linson--Ginzburg--Soergel~\cite{bgs} in the setting of the Kazhdan--Lusztig conjecture on characters of simple highest weight modules for a complex semisimple Lie algebra. In this paper the authors explained in particular the relation between this construction and some ``mixed'' properties of $\ell$-adic perverse sheaves on flag varieties of reductive algebraic groups. A modified form of this Koszul duality for constructible sheaves on flag varieties was later generalized to Kac--Moody groups by Bezrukavnikov--Yun~\cite{by}, which allowed to explain the relations between several equivalences of categories constructed by Bezrukavnikov with various collaborators, and related to local geometric Langlands duality and representations of quantum groups at roots of unity.

More recently, as part of a program aiming at generalizing some of Bezrukavni\-kov's equivalences in the setting of reductive algebraic groups over fields of positive characteristic, a version of the Koszul duality of~\cite{by} was obtained by the first author with Achar, Makisumi and Williamson in~\cite{amrw2}. This construction is more formal and less geometric than that of~\cite{by}; in particular it involves the ``mixed derived categories'' constructed with Achar (see~\cite{ar}), which are useful but rather ad-hoc. A more natural setting for this construction seems to be that of the Hecke category attached to a Coxeter system and a given realization by Elias--Williamson~\cite{ew} and, although the final construction of Koszul duality was restricted to the case of Cartan realizations of crystallographic Coxeter systems, i.e.~the case when the Hecke category can be described geometrically in terms of parity complexes on flag varieties of Kac--Moody groups, part of the constructions involved were already treated in full generality in~\cite{amrw1}.

This raised the question of the existence of a version of Koszul duality in the general setting of~\cite{ew}, involving in particular a general Coxeter system.\footnote{A first suggestion of the existence of such a construction can be found in~\cite[Remark~3.5]{ew}.} The main result of this paper is a realization of this idea based on some prior work with Achar~\cite{arv}, under some technical conditions that we discuss in~\S\ref{ss:intro-assumptions} below.

\subsection{Statement}

Let us consider a Coxeter system $(W,S)$ and a realization $\fh$ of $(W,S)$ over a field $\bk$ satisfying appropriate assumptions (see~\S\ref{ss:intro-assumptions}). To these data Elias--Williamson attach a $\bk$-linear monoidal category $\DiagBS(\fh,W)$ endowed with a ``shift'' autoequivalence $(1)$, defined by generators and relations, and whose split Grothen\-dieck group identifies with the Hecke algebra of $(W,S)$. For any objects $X$ and $Y$ in $\DiagBS(\fh,W)$, the $\bk$-vector space
\[
 \bigoplus_{n \in \Z} \Hom_{\DiagBS(\fh,W)}(X,Y(n))
\]
admits a canonical structure of graded bimodule over the symmetric algebra $R$ of $V^*$, where $V$ is the representation underlying $\fh$, and by ``killing'' the left, resp.~right, action of this algebra one obtains a category $\oDiagBS(\fh,W)$, resp.~$\uDiagBS(\fh,W)$. We then introduce the ``biequivariant'', ``right equivariant'' and ``left equivariant'' categories attached to $(\fh,W)$ as
\begin{align*}
 \BE(\fh,W) &= \Kb \DiagBS^\oplus(\fh,W), \\
 \RE(\fh,W) &= \Kb \oDiagBS^\oplus(\fh,W), \\ 
 \LE(\fh,W) &= \Kb \uDiagBS^\oplus(\fh,W)
\end{align*}
where the subscript ``$\oplus$'' indicates the additive hull.
(The terminology, taken from~\cite{amrw1}, is motivated by the special case when the Hecke category can be described in terms of constructible sheaves: the biequivariant category involves sheaves on the group which are equivariant for a Borel subgroup on both sides, while the right, resp.~left, equivariant category involves sheaves which are equivariant for the action on the right, resp.~left.)

With this notation, in the special case considered there, one form of the Koszul duality of~\cite{amrw2} is an equivalence of triangulated categories
\[
\kappa:\RE(\fh,W)\simto\LE(\fh^*,W) 
\]
which satisfies $\kappa \circ (1) = (-1)[1] \circ \kappa$, where $\fh^*$ is the dual realization (obtained by switching roots and coroots; in the case related to geometry this amounts to Langlands duality). To state further properties of this equivalence one needs to recall that the categories $\RE(\fh,W)$ and $\LE(\fh^*,W)$ admit canonical ``perverse'' t-structures (again, the terminology comes from geometry) whose hearts admit canonical highest weight structures. More precisely, at the time when~\cite{amrw1,amrw2} were written this was known only in the case of Cartan realizations of crystallographic Coxeter systems, but in the meantime this construction was extended to the general setting in~\cite{arv}. As in any highest weight category one can consider the indecomposable tilting objects in these categories, and the main property of Koszul duality can be roughly stated as the fact that it exchanges the indecomposable objects in the karoubian envelope of $\oDiagBS^\oplus(\fh,W)$, resp.~$\uDiagBS^\oplus(\fh^*,W)$, with the indecomposable tilting objects in the heart of the perverse t-structure on $\LE(\fh^*,W)$, resp.~$\RE(\fh,W)$. One of the main results of the present paper is a version of this statement for general realizations of general Coxeter groups, see Theorem~\ref{thm:self-duality}.

\subsection{Strategy}

The main strategy of our proof is similar to that used in~\cite{amrw2} (which, itself, is an adaptation of constructions considered in~\cite{soergel-perv,bbm,bgs,by}). Namely, instead of constructing the equivalence $\kappa$ directly, we first consider a ``free monodromic'' variant, from which (the inverse of) $\kappa$ will be obtained by essentially ``killing'' a left $R$-action on morphisms. The main point is that in this setting one works with \emph{monoidal} categories, where the definition of $\DiagBS(\fh,W)$ by generators and relations can be used with great effect. 

We therefore construct a category of ``free monodromic tilting objects'' $\TBS(\fh,W)$ attached to $(\fh,W)$, and then an equivalence of monoidal categories
\[
 \DiagBS(\fh^*,W) \simto \TBS(\fh,W),
\]
see Theorem~\ref{thm:Phi}. The main difficulty lies in the \emph{definition} of such a functor; once this is known the same arguments as in~\cite{amrw2} can be developed to prove that it is an equivalence. To define this functor, in view of the definition of $\DiagBS(\fh^*,W)$ we need to describe the images of the generating morphisms, and then check that these images satisfy the relations imposed in $\DiagBS(\fh^*,W)$. The construction of all the morphisms involving only one color can be copied from~\cite{amrw2}, but the definition of the morphism corresponding the $2$-colors generators given there relies partially on geometry. Here we give a different (and general) construction of this morphism in Section~\ref{sec:construction}. 

This construction relies on the prior construction of a ``functor $\V$'' in case $W$ is finite, explained in Section~\ref{sec:functor V}. This construction is similar to that in~\cite{amrw2}, with one notable exception: in~\cite{amrw2} this functor takes values in ``usual'' Soergel bimodules, which leads to imposing technical assumptions on the characteristic of $\bk$; these assumptions can be removed later in the paper using some change-of-scalars arguments which make sense only for Cartan realizations. Here we use a variant of the construction of Soergel bimodules developed in the meantime by Abe~\cite{abe1}, which allows to avoid these technical assumptions completely.

The proof that these morphisms satisfy the required relations is again similar to the corresponding part of~\cite{amrw2}, but using Abe's category of bimodules rather than plain bimodules.

\subsection{Assumptions}
\label{ss:intro-assumptions}

The assumptions that we have to impose on our relization $\fh$ are explained in detail in~\S\ref{ss:convention}, \S\ref{ss:dual-real} and~\S\ref{ss:bifunctoriality}. The assumptions of~\S\S\ref{ss:convention}--\ref{ss:dual-real} are ``standard'' assumptions that are required for the theories in~\cite{ew} and~\cite{abe1} to behave appropriately. They are known to hold in the main examples of realizations that arise ``in nature,'' i.e.~the Cartan realizations of crystallographic Coxeter systems and the geometric realization (and its variants considered by Soergel) of any Coxeter system not involving type $\mathsf{H}_3$.

The assumption of~\S\ref{ss:bifunctoriality} is of a different kind: in~\cite{amrw1} an ad-hoc version of the ``free monodromic completed category'' of~\cite{by} was constructed in the diagrammatic setting. This category (or a subcategory) should be monoidal, and all the structures involved can indeed be constructed, but the question of whether these structures satisfy the appropriate ``interchange law'' was left open. Here we assume that this property is satisfied for appropriate objects. A proof that it is indeed the case in full generality has been announced by Hogancamp and Makisumi, but no written account of their work is available as of now. In~\cite{amrw1} it was proved that this property holds for Cartan realizations of crystallographic Coxeter systems, and here we remark that the same approach, combined with the results of~\cite{arv}, also applies under some technical assumptions that are satisfied for geometric and Soergel realizations. 

In particular, the presently available literature is enough to show that all of our statements hold at least for the 2 main families of examples of realizations of Coxeter systems that are known.

\subsection{Application}

As an application, we show in~\S\ref{ss:application} that the combinatorics of the indecomposable tilting objects in $\RE(\fh,W)$ (constructed in~\cite{arv}) is governed by the combinatorics of the ``canonical basis'' attached to the dual realization. In particular, in the case of Soergel realizations, this combinatorics is governed by Kazhdan--Lusztig polynomials (see Example~\ref{ex:t-poly-Soergel}), and the heart of the perverse t-structure on $\RE(\fh,W)$ is equivalent to the category of finite-dimensional graded modules over a finite-dimensional Koszul ring (see Remark~\ref{rmk:koszul-duality}).

\subsection{Acknowledgements}

The results in this paper are the realization of a project initiated with Pramod Achar, Shotaro Makisumi and Geordie Williamson. We thank them for very helpful discussions at various stages of its completion, and for their contributions in the formulation of several key ideas. This work also owes much to the work of Noriyuki Abe~\cite{abe1,abe2}, which provided the necessary ingredients to generalize the approach used in~\cite{amrw2}. We thank an anonymous referee for helpful comments.

A large part of our work was accomplished during a visit of the second author in Clermont-Ferrand funded by CIMPA-ICTP fellowships program ``Research in Pairs'' and ERC.

\section{Preliminaries}
\label{sec:prelim}

\subsection{Conventions and assumptions}
\label{ss:convention}

Throughout the work $(W,S)$ denotes a Coxeter system with $S$ finite. (As usual, we will usually only indicate $W$ in our notations, although all the structures we consider also depend on the choice of Coxeter generators $S$.) We consider on $W$ the Bruhat order $\leq$ and the length function $\ell$. We will use the standard terminology regarding Coxeter systems, as recalled e.g.~in~\cite[\S3.1]{arv}. In particular, an {\it expression} is a word $\uw=(s_1, \dots, s_n)$ in $S$ and $\pi(\uw)=s_1\cdots s_n$ denotes the element in $W$ expressed by $\uw$. The set of expressions will be denoted $\Exp(W)$. The \emph{length} of an expression $\uw$ is its length as a word; it will be denoted $\ell(\uw)$. We will identify simple reflections with the corresponding $1$-letter expression whenever convenient. Given a pair $(s,t)$ of simple reflections, we will denote by $m_{s,t} \in \Z_{\geq 1} \cup \{\infty\}$ the order of $st$ in $W$, and by $\langle s,t \rangle$ the subgroup generated by $s$ and $t$.

We fix a field $\bk$ and a realization
\[
\fh= \bigl( V, (\alpha_s^\vee:s\in S), (\alpha_s:s\in S) \bigr)
\]
of $(W,S)$ over $\ku$ in the sense of Elias--Williamson~\cite[Definition~3.1]{ew}. In particular, $V$ is a finite-dimensional $\bk$-vector space, $(\alpha_s^\vee:s\in S)$ is a collection of vectors in $V$, $(\alpha_s:s\in S)$ is a collection of vectors in $V^*:=\Hom_\bk(V,\bk)$, and there exists an action of $W$ on $V$ such that for $s \in S$ and $v \in V$ we have
\[
s \cdot v=v-\langle\alpha_s,v\rangle\alpha_s^\vee.
\]
Realizations can be restricted to parabolic subsystems of $(W,S)$, by simply forgetting part of the elements $\alpha_s$ and $\alpha_s^\vee$: if $S' \subset S$ is a subset, and $W' \subset W$ is the subgroup generated by $S'$, we will denote by $\fh_{|W'}$ the realization $( V, (\alpha_s^\vee:s\in S'), (\alpha_s:s\in S') )$ of $(W',S')$.

In addition to the conditions appearing in this definition, it has been recently explained in~\cite{ew20} that some further restrictions have to be imposed in order for the theory developed in~\cite{ew} to behave as expected, most of which were made more explicit in~\cite{hazi}. Here we will assume that the following conditions are satisfied.
\begin{enumerate}
\item
\label{it:assumptions-1}
The realization is balanced (see~\cite[Definition~3.7]{ew}) and satisfies Demazure surjectivity (see~\cite[Assumption~3.9]{ew}).
\item
\label{it:assumptions-3}
In case $W$ admits a parabolic subgroup of type $\mathsf{H}_3$, we assume that there exists a linear combination of diagrams as in~\cite[Equation~(5.12)]{ew} which is sent to $0$ under the operation described in~\cite[\S 2]{ew20}. 
\item
\label{it:assumptions-4}
For any pair $(s,t)$ of distinct simple reflections such that $m_{s,t} < \infty$ we have
\[
\genfrac{[}{]}{0pt}{}{m_{s,t}}{k}_{s}(\langle \alpha_s^\vee, \alpha_t \rangle, \langle \alpha_t^\vee, \alpha_s \rangle) = 
\genfrac{[}{]}{0pt}{}{m_{s,t}}{k}_{t}(\langle \alpha_s^\vee, \alpha_t \rangle, \langle \alpha_t^\vee, \alpha_s \rangle) = 0
\]
for all integers $1 \leq k \leq m_{st}-1$, where we use the notation of~\cite{hazi}.
\end{enumerate}
The assumptions in~\eqref{it:assumptions-1} are standard in this theory. The second one is really necessary; the first one can usually be relaxed at the cost of a more complicated combinatorics (see e.g.~\cite[\S 7]{ew20}), but we will not consider this question here. Assumptions~\eqref{it:assumptions-3} and~\eqref{it:assumptions-4} are also necessary for the theory of~\cite{ew}, hence also for all of its applications, although this was not made explicit before~\cite{ew20,hazi}. (In particular, they should be imposed in~\cite{amrw1} and in~\cite{arv}.) Here, by the main result of~\cite{hazi},~\eqref{it:assumptions-4} implies the existence and rotatability of Jones--Wenzl projectors, which as explained in~\cite{ew20} plays a crucial role in this story. Note that~\eqref{it:assumptions-4} is also the technical condition imposed in~\cite{abe2} to ensure that the theory of~\cite{abe1} applies.


It is important to note that if the assumptions~\eqref{it:assumptions-1}--\eqref{it:assumptions-4} are satisfied by a realization, then they are satisfied by its restriction to any parabolic subsystem of $(W,S)$.
Further assumptions will be imposed and discussed in~\S\ref{ss:dual-real} and~\S\ref{ss:bifunctoriality}; they are also stable under restriction to a parabolic subsystem.

\begin{example}
\label{ex:realizations}
The main examples of data as above the reader can keep in mind are the following.

\begin{enumerate}
\item
\label{it:Cartan-real}
Let $A=(a_{i,j})_{i,j \in I}$ be a generalized Cartan matrix. A \emph{Kac--Moody root datum} associated with $A$ is a triple
\[
(\mathbf{X}, (\alpha_i : i \in I), (\alpha_i^\vee : i \in I))
\]
where $\mathbf{X}$ is a finite free $\Z$-module, $(\alpha_i : i \in I)$ is a family of elements of $\mathbf{X}$, and $(\alpha_i ^\vee: i \in I)$ is a family of elements of $\Hom_{\Z}(\mathbf{X},\Z)$, such that $\langle \alpha_i^\vee, \alpha_j \rangle = a_{i,j}$ for any $i,j \in I$. To $A$ one can associate a Coxeter system $(W,S)$ where $S$ is in bijection with $I$ (through a map $s \mapsto i_s$), and the order $m_{s,t}$ of $st$ is determined as follows: 
\[
m_{s,t} = \begin{cases}
2 & \text{if $a_{i_s i_t} a_{i_t i_s} =0$;} \\
3 & \text{if $a_{i_s i_t} a_{i_t i_s} =1$;} \\
4 & \text{if $a_{i_s i_t} a_{i_t i_s} =2$;} \\
6 & \text{if $a_{i_s i_t} a_{i_t i_s} =3$;} \\
\infty & \text{if $a_{i_s i_t} a_{i_t i_s} \geq 4$.} \\
\end{cases}
\]
To $(\mathbf{X}, (\alpha_i : i \in I), (\alpha_i^\vee : i \in I))$ one can associate a realization of $(W,S)$ over any field $\bk$ by setting $V:=\bk \otimes_{\Z} \Hom_{\Z}(\mathbf{X},\Z)$, and choosing for $(\alpha_s^\vee : s \in S)$ and $(\alpha_s : s \in S)$ the images of $(\alpha_i^\vee : i \in I)$ and $(\alpha_i : i \in I)$ in $V$ and $V^*$ respectively.

Such realizations are called \emph{Cartan realizations of crystallographic Coxeter groups}.
The status of our assumptions above for such realizations is discussed at length in~\cite[Chap.~II, \S 2.4 and~\S 3.2]{curso}. Regarding~\eqref{it:assumptions-1}, such realizations are always balanced; they satisfy Demazure surjectivity at least when $\mathrm{char}(\bk) \neq 2$. 
Assumption~\eqref{it:assumptions-3} is irrelevant since crystallographic Coxeter groups do not admit parabolic subgroups of type $\mathsf{H}_3$. Assumption~\eqref{it:assumptions-4} is automatically satisfied.

Cartan realizations of crystallographic Coxeter groups are the realizations considered in~\cite[Chap.~10--11]{amrw1} and~\cite{amrw2} (where more general coefficient rings are allowed.)
For such a realization, the associated Hecke category (see~\S\ref{ss:EW-category}) can be realized geometrically as a category of parity complexes on the flag variety of the Kac--Moody group associated with $(\mathbf{X}, (\alpha_i : i \in I), (\alpha_i^\vee : i \in I))$; see~\cite[Part~III]{rw} for details.
\item
\label{it:Cartan-real-finite}
In the special case when $A$ is a Cartan matrix, the datum of a Kac--Moody root datum of $A$ is equivalent to the datum of a based root datum with associated Cartan matrix $A$. In this case, $W$ is the Weyl group of the associated connected reductive algebraic group (over any algebraically closed field).
\item
\label{it:geom-real}
Let now $(W,S)$ be an arbitrary Coxeter system with $S$ finite. Let $V$ be the associated \emph{geometric representation} of $W$; it is a representation over $\mathbb{R}$, and comes with a basis $(e_s : s \in S)$ indexed by $S$ and a bilinear form $\langle -,- \rangle$. One can ``upgrade'' this representation to a realization of $(W,S)$ over $\mathbb{R}$, called the geometric realization, by setting $\alpha_s^\vee := e_s$ and $\alpha_s := 2 \langle e_s, - \rangle$. As explained in~\cite[Chap.~II, \S 2.4 and~\S 3.2]{curso}, for this realization our assumptions~\eqref{it:assumptions-1}
and~\eqref{it:assumptions-4} are satisfied. The status of assumption~\eqref{it:assumptions-3} (in case $(W,S)$ has a parabolic subsystem of type $\mathsf{H}_3$) is unclear to us.
\item
\label{it:Soergel-real}
For an arbitrary Coxeter system $(W,S)$ with $S$ finite, one can also consider variants of the geometric realization considered by Soergel in~\cite{soergel-bim}, see e.g.~\cite[Chap.~II, \S 1.2.2]{curso}. Namely, consider a vector space $V$ endowed with linearly independent families $(e_s : s \in S)$ of vectors of $V$ and $(e_s^* : s \in S)$ of vectors of $V^*$ such that
\[
 \langle e_t, e_s^* \rangle = -2\cos \left( \frac{\pi}{m_{s,t}} \right)
\]
where $m_{s,t}$ is the order of $st$ in $W$. (We use the convention that $\frac{\pi}{\infty}=0$. Note also that such data always exist.) Then $(V,(e_s : s \in S),(e_s^* : s \in S))$ is a realization of $(W,S)$,
see~\cite[Chap.~II, Remark~2.7]{curso}. These realizations will be called Soergel realizations. (Note that in case $W$ is finite, the geometric realization is an example of a Soergel realization.) In this case again, our assumptions~\eqref{it:assumptions-1}
and~\eqref{it:assumptions-4} are satisfied (see~\cite[Chap.~II, \S 2.4 and~\S 3.2]{curso}), but the status of assumption~\eqref{it:assumptions-3} (in case $(W,S)$ has a parabolic subsystem of type $\mathsf{H}_3$) is unclear to us. For such a realization, the Hecke category is equivalent to the corresponding category of Soergel bimodules by~\cite[\S 6.7]{ew}.
\end{enumerate}
\end{example}


\subsection{Gradings}
\label{ss:gradings}

By a ``graded'' (resp. ``bigraded'') vector space, we mean a $\Z$-graded (resp. $\Z^2$-graded) vector space. Whenever convenient, we will identify graded vector spaces with bigraded vector spaces which are zero in all degrees belonging to $(\Z \smallsetminus \{0\}) \times \Z$. If $M$ is graded, resp.~bigraded, its component in degree $n$, resp.~in bidegree $(n,m)$, will be denoted $M_n$, resp.~$M^n_m$. The shift-of-grading functor $(1)$ on a 
bigraded vector space $M=\oplus_{i,j\in\Z}M_j^i$ is defined by
\[
M(1)^i_j=M_{j+1}^{i+1}.
\]
We also have shift functors $[1]$ and $\langle 1\rangle:=(-1)[1]$, which satisfy $(M[1])^i_j=M_j^{i+1}$ and $(M\langle1\rangle)^i_j=M_{j-1}^{i}$. Note that $\langle1\rangle$ stabilizes graded vectors spaces.

We will work in particular with the symmetric algebras
\[
R:=\operatorname{Sym}(V^*), \quad
R^\vee:=\operatorname{Sym}(V), \quad
R^\wedge:=\operatorname{Sym}(V),
\]
considered as graded rings where $V^*\subset R$ is in degree $2$, $V\subset R^\vee$ is in degree $-2$ and $V\subset R^\wedge$ is in degree $2$. We will also consider the localization $Q$ of the ring $R$ with respect to the multiplicative subset generated by $\{w(\alpha_s):s\in S, \, w\in W\}$; this ring has a
natural grading where the elements $\frac{1}{w(\alpha_s)}$ are in degree $-2$. Analogously, we denote by $Q^\vee$ and $Q^\wedge$ the corresponding localizations of $R^\vee$ and $R^\wedge$.

Let $R^\vee\mhyphen\Mod^\Z\mhyphen R^\vee$ denote the category of graded $R^\vee$-bimodules, and define analogously $R^\wedge\mhyphen\Mod^\Z\mhyphen R^\wedge$. Then $\langle1\rangle$ induces autoequivalences of these categories, which will be denoted similarly. If $M$ belongs to $R^\vee\mhyphen\Mod^\Z\mhyphen R^\vee$, we let $M^\wedge$  be the graded $R^\wedge$-bimodule which is $M$ as ungraded bimodule and whose homogeneous components are $(M^\wedge)_i=M_{-i}$, $i\in\Z$. This induces a functor from $R^\vee\mhyphen\Mod^\Z\mhyphen R^\vee$ to $R^\wedge\mhyphen\Mod^\Z\mhyphen R^\wedge$ which satisfies
\begin{equation}
\label{eqn:wedge-shift}
(M \langle 1 \rangle)^\wedge = M^\wedge \langle -1 \rangle.
\end{equation}

Given a graded free right $R^\vee$-module, resp.~vector space, of finite rank $M\simeq\oplus_iR^\vee \langle n_i \rangle$, resp.~$M\simeq\oplus_i\bk \langle n_i\rangle$, we set
\[
\grk_{R^\vee}M:=\sum_i v^{-n_i}, \quad \text{resp.} \quad \grk_{\bk}M:=\sum_i v^{-n_i},
\]
considered as elements in $\Z[v,v^{-1}]$ where $v$ is an indeterminate. (In other words, if $V$ is a finite-dimensional $\bk$-vector space we have $\grk(V)=\sum_n \dim(V_{-n})v^n$.) Of course, if $M,N$ are graded free right $R^\vee$-module of finite rank, then $M\simeq N$ if and only if $\grk_{R^\vee} M=\grk_{R^\vee} N$. We define analogously the function $\grk_{R^\wedge}$. Note that if $M$ is a graded free right $R^\vee$-module, then $M^\wedge$ is a graded free right $R^\wedge$-module and we have
\begin{equation}
\label{eqn:grk-wedge}
\grk_{R^\wedge}M^\wedge=\overline{\grk_{R^\vee}M}
\end{equation}
where $\overline{\cdot}$ is the unique ring automorphism of $\Z[v,v^{-1}]$ such that $\overline{v}=v^{-1}$.

Below we will need the following application of the graded Nakayama lemma, where we consider $\bk$ (concentrated in degree $0$) as a graded $R^\vee$-module by letting $V$ act by zero.

\begin{lema}
\label{le:technical}
Let $M$ and $N$ be graded free right $R^\vee$-modules of finite rank and $f:M\to N$ be a morphism of graded $R^\vee$-modules. If 
\[
f\otimes_{R^\vee}\bk: M\otimes_{R^\vee}\bk \to N\otimes_{R^\vee}\bk
\]
is injective (resp.~surjective), then $f$ is injective (resp.~surjective).
\end{lema}

\begin{proof}
Let $C$ denote the complex
\[
\cdots 0 \to M\xrightarrow{f} N \to 0\cdots
\]
with $M$, resp.~$N$, placed in degree $-1$, resp.~$0$. If $f\otimes_{R^\vee} \bk$ is injective, this complex satisfies the assumptions of~\cite[Lemma 3.4.1]{amrw1}. In particular, it follows that $H^{-1}(C)=0$, that is $f$ is injective. If $f\otimes_{R^\vee} \bk$ is surjective, we use \cite[Lemma 3.4.1]{amrw1} with the same complex with $M$, resp.~$N$, placed in degree $0$, resp.~$1$, to deduce that  $f$ is also surjective.
\end{proof}

\subsection{Dual realization}
\label{ss:dual-real}

We will also consider the dual realization
\[
\fh^*=(V^*, (\alpha_s:s\in S), (\alpha_s^\vee:s\in S))
\]
of $(W,S)$ over $\bk$. 
It is clear that this realization satisfies our assumptions~\eqref{it:assumptions-1}
and~\eqref{it:assumptions-4}. It is not clear (to us) whether assumption~\eqref{it:assumptions-3} is automatically satisfied; in doubt, we will assume that it is also satisfied by $\fh^*$. Note that the graded ring playing, with respect to $\fh^*$, the role that $R$ plays for $\fh$ is $R^\wedge$.


\begin{example}
 \phantomsection
 \label{ex:dual-real}
 \begin{enumerate}
  \item In the setting of Example~\ref{ex:realizations}\eqref{it:Cartan-real}, if $(\mathbf{X}, (\alpha_i : i \in I), (\alpha_i^\vee : i \in I))$ is a Kac--Moody root datum associated with a generalized Cartan matrix $A$, then
  \[
   (\Hom_{\Z}(\mathbf{X},\Z), (\alpha_i^\vee : i \in I), (\alpha_i : i \in I))
  \]
is a Kac--Moody root datum associated with the generalized Cartan matrix ${}^{\mathrm{t}} \hspace{-1pt} A$. These generalized Cartan matrices share the same associated Coxeter system $(W,S)$, and for any field $\bk$ the dual of the realization over $\bk$ associated with $(\mathbf{X}, (\alpha_i : i \in I), (\alpha_i^\vee : i \in I))$ is the realization over $\bk$ associated with $(\Hom_{\Z}(\mathbf{X},\Z), (\alpha_i^\vee : i \in I), (\alpha_i : i \in I))$. Note that this ``duality'' of Kac--Moody root data restricts to Langlands' duality in the setting of Example~\ref{ex:realizations}\eqref{it:Cartan-real-finite}.
\item
\label{it:dual-soergel-real}
Consider the setting of Example~\ref{ex:realizations}\eqref{it:Soergel-real}. By definition the dual of a Soergel realization is also a Soergel realization. In particular, in case $W$ is finite, the geometric realization of Example~\ref{ex:realizations}\eqref{it:geom-real} is self dual.
 \end{enumerate}
\end{example}

\section{Two incarnations of the Hecke category}
\label{sec:soergel}

We continue with our data $(W,S)$ and $\fh$ as in Section~\ref{sec:prelim}, which satisfy the conditions of~\S\ref{ss:convention} and~\S\ref{ss:dual-real}.
We recall in this section the definitions of the Elias--Williamson diagrammatic category and of Abe's category attached to $\fh$ and $(W,S)$.

\subsection{The Elias--Williamson diagrammatic category}
\label{ss:EW-category}

We will denote by
\[
\mathscr{D}_{\mathrm{BS}}(\fh,W)
\]
the category attached to $(W,S)$ and $\fh$ in~\cite[\S 5.2]{ew} (see also~\cite[Chap.~II, \S 2.5]{curso}). This category is a $\bk$-linear monoi\-dal category, which can be considered equivalently as enriched over graded vector spaces or endowed with a shift-of-grading autoequivalence $(1)$, see~\cite[\S 2.1]{arv}. From the first of these points of view, the objects in $\mathscr{D}_{\mathrm{BS}}(\fh,W)$ are the symbols $B_{\uw}$ for $\uw \in \Exp(W)$. The monoidal product is defined by $B_{\uv} \star B_{\uw} = B_{\uv\uw}$ where $\uv\uw$ is the concatenation of $\uv$ and $\uw$. The morphisms are generated (under horizontal and vertical concatenation, and $\bk$-linear combinations) by morphisms depicted by some diagrams recalled below, and are subject to a number of relations for which we refer to~\cite{ew},~\cite[\S 2.3]{amrw1} or~\cite{curso}. The generating morphisms are (to be read from bottom to top):
\begin{enumerate}
\item
for any homogeneous $f \in R$, a morphism
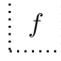
\[
    \begin{tikzpicture}[thick,scale=0.07,baseline]
      \node at (0,0) {$f$};
      \draw[dotted] (-5,-5) rectangle (5,5);
    \end{tikzpicture}
\]
from $B_\varnothing$ to itself, of degree $\deg(f)$;
\item
for any $s \in S$, ``dot'' morphisms
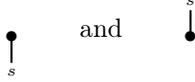
\[
       \begin{tikzpicture}[thick,scale=0.07,baseline]
      \draw (0,-5) to (0,0);
      \node at (0,0) {$\bullet$};
      \node at (0,-6.7) {\tiny $s$};
    \end{tikzpicture}
    \qquad \text{and} \qquad
      \begin{tikzpicture}[thick,baseline,xscale=0.07,yscale=-0.07]
      \draw (0,-5) to (0,0);
      \node at (0,0) {$\bullet$};
      \node at (0,-6.7) {\tiny $s$};
    \end{tikzpicture}
\]
from $B_s$ to $B_\varnothing$ and from $B_\varnothing$ to $B_s$ respectively, of degree $1$;
\item
for any $s \in S$, trivalent morphisms
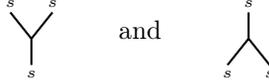
\[
        \begin{tikzpicture}[thick,baseline,scale=0.07]
      \draw (-4,5) to (0,0) to (4,5);
      \draw (0,-5) to (0,0);
      \node at (0,-6.7) {\tiny $s$};
      \node at (-4,6.7) {\tiny $s$};
      \node at (4,6.7) {\tiny $s$};      
    \end{tikzpicture}
    \qquad \text{and} \qquad
        \begin{tikzpicture}[thick,baseline,scale=-0.07]
      \draw (-4,5) to (0,0) to (4,5);
      \draw (0,-5) to (0,0);
      \node at (0,-6.7) {\tiny $s$};
      \node at (-4,6.7) {\tiny $s$};
      \node at (4,6.7) {\tiny $s$};    
    \end{tikzpicture}
\]
from $B_s$ to $B_{(s,s)}$ and from $B_{(s,s)}$ to $B_s$ respectively, of degree $-1$;
\item
for any pair $(s,t)$ of distinct simple reflections such that $st$ has finite order $m_{st}$ in $W$, a morphism
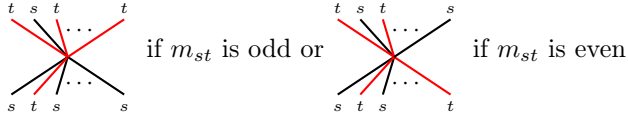
\[
    \begin{tikzpicture}[yscale=0.5,xscale=0.3,baseline,thick]
\draw (-2.5,-1) to (0,0) to (-1.5,1);
\draw (-0.5,-1) to (0,0);
\draw[red] (-1.5,-1) to (0,0) to (-2.5,1);
\draw[red] (0,0) to (-0.5,1);
\node at (-2.5,-1.3) {\tiny $s$\vphantom{$t$}};
\node at (-1.5,1.3) {\tiny $s$\vphantom{$t$}};
\node at (-0.5,-1.3) {\tiny $s$\vphantom{$t$}};
\node at (-1.5,-1.3) {\tiny $t$};
\node at (-2.5,1.3) {\tiny $t$};
\node at (-0.5,1.3) {\tiny $t$};
\node at (0.6,-0.7) {\small $\cdots$};
\node at (0.6,0.7) {\small $\cdots$};
\draw (2.5,-1) -- (0,0);
\draw[red] (2.5,1) -- (0,0);
\node at (2.5,-1.3) {\tiny $s$\vphantom{$t$}};
\node at (2.5,1.3) {\tiny $t$};
\end{tikzpicture} \text{ if $m_{st}$ is odd or}
    \begin{tikzpicture}[yscale=0.5,xscale=0.3,baseline,thick]
\draw (-2.5,-1) to (0,0) to (-1.5,1);
\draw (-0.5,-1) to (0,0);
\draw[red] (-1.5,-1) to (0,0) to (-2.5,1);
\draw[red] (0,0) to (-0.5,1);
\node at (-2.5,-1.3) {\tiny $s$\vphantom{$t$}};
\node at (-1.5,1.3) {\tiny $s$\vphantom{$t$}};
\node at (-0.5,-1.3) {\tiny $s$\vphantom{$t$}};
\node at (-1.5,-1.3) {\tiny $t$};
\node at (-2.5,1.3) {\tiny $t$};
\node at (-0.5,1.3) {\tiny $t$};
\node at (0.6,-0.7) {\small $\cdots$};
\node at (0.6,0.7) {\small $\cdots$};
\draw[red] (2.5,-1) -- (0,0);
\draw (2.5,1) -- (0,0);
\node at (2.5,-1.3) {\tiny $t$};
\node at (2.5,1.3) {\tiny $s$\vphantom{$t$}};
\end{tikzpicture} \text{ if $m_{st}$ is even}
\]
from $B_{(s,t,\cdots)}$ to $B_{(t,s,\cdots)}$ (where each expression has length $m_{st}$, and colors alternate), of degree $0$.
\end{enumerate}
The graded vector space of morphisms from $B_{\uw}$ to $B_{\uv}$ will be denoted
\[
\Hom_{\DiagBS(\fh,W)}^\bullet(B_{\uw},B_{\uv}).
\]
Considering $\DiagBS(\fh,W)$ as a category with shift-of-grading autoequivalence $(1)$, its objects are the symbols $B_{\uw}(n)$ where $\uw \in \Exp(W)$ and $n \in \Z$, and the vector space of morphisms from $B_{\uw}(n)$ to $B_{\uv}(m)$ is $\Hom_{\DiagBS(\fh,W)}^{m-n}(B_{\uw},B_{\uv})$. This is the point of view we will mostly use below.

More generally, given $X,Y$ in $\DiagBS(\fh,W)$ we will set
\[
\Hom_{\DiagBS(\fh,W)}^\bullet(X,Y) := \bigoplus_{n \in \Z} \Hom_{\DiagBS(\fh,W)}(X,Y(n)).
\]
This graded vector space has a natural structure of
graded $R$-bimodule, where the left (resp.~right) action of $f \in R_n$ is given by adding a box labelled by $f$ to the left (resp.~right) of a diagram. With this structure, $\Hom_{\DiagBS(\fh,W)}^\bullet(X,Y)$ is graded free of finite rank as a left $R$-module and as a right $R$-module, see~\cite[Corollary~6.14]{ew}.

We will denote by $\DiagBS^\oplus(\fh,W)$ the additive hull of $\DiagBS(\fh,W)$, and by $\Diag(\fh,W)$ the karoubian envelope of $\DiagBS^\oplus(\fh,W)$. The latter category is Krull--Schmidt, and there exists a family $(B_w : w \in W)$ of objects in $\DiagBS(\fh,W)$ characterized in~\cite[Theorem 6.26]{ew} and such that the assignment $(w,n) \mapsto B_w(n)$ induces a bijection between $W \times \Z$ and the set of isomorphism classes of indecomposable objects in $\Diag(\fh,W)$.

We will also consider the category $\oDiagBS(\fh,W)$ which has the same objects as $\DiagBS(\fh,W)$, and such that the morphism space from $X$ to $Y$ is the subspace of degree-$0$ elements in the graded vector space
\[
\bk \otimes_R \Hom_{\DiagBS(\fh,W)}^\bullet(X,Y)
\]
(where $\bk$ is in degree $0$ and $R$ acts via the quotient $R/V \cdot R=\bk$).
The functor $(1)$ induces an autoequivalence of $\oDiagBS(\fh,W)$ which will be denoted similarly, and $\oDiagBS(\fh,W)$ is naturally a right module category for the monoidal category $\DiagBS(\fh,W)$. We have a natural functor
\begin{equation}
\label{eqn:Diag-oDiag}
\DiagBS(\fh,W) \to \oDiagBS(\fh,W);
\end{equation}
the image of $B_{\uw}$ under this functor will be denoted $\oB_{\uw}$.
As for $\DiagBS(\fh,W)$, we will denote by $\oDiagBS^\oplus(\fh,W)$ the additive hull of $\oDiagBS(\fh,W)$, and by $\oDiag(\fh,W)$ the karoubian envelope of $\oDiagBS^\oplus(\fh,W)$.

The functor~\eqref{eqn:Diag-oDiag} induces a functor $\DiagBS^\oplus(\fh,W) \to \oDiagBS^\oplus(\fh,W)$, and then a functor $\Diag(\fh,W) \to \oDiag(\fh,W)$. The category $\oDiag(\fh,W)$ is Krull--Schmidt, being karoubian and with finite-dimensional morphism spaces, see~\cite[Corollary~A.2]{cyz}. For $w \in W$, we will denote by $\oB_w$ the image of $B_w$ in $\oDiag(\fh,W)$. Then $\End_{\oDiag(\fh,W)}(\oB_w)$ is a quotient of $\End_{\Diag(\fh,W)}(B_w)$, hence is a local ring, which implies that $\oB_w$ is an indecomposable object. Using this, it is easily seen that the assignment $(w,n) \mapsto \oB_w(n)$ induces a bijection between $W \times \Z$ and the set of isomorphism classes of indecomposable objects in $\oDiag(\fh,W)$.

We will also denote by $\uDiagBS(\fh,W)$, $\uDiagBS^\oplus(\fh,W)$ and $\uDiag(\fh,W)$ the categories obtained in the same way using the tensor product on the \emph{right} with $\bk$ over $R$. Here $\uDiagBS(\fh,W)$ is naturally a \emph{left} module category over $\DiagBS(\fh,W)$, and the same considerations as above for $\oDiagBS(\fh,W)$ apply. We will use obvious variants of the notations introduced above; in particular, for $w \in W$, resp.~for $\uw \in \Exp(W)$, we define the object $\uB_w \in \uDiag(\fh,W)$, resp.~$\uB_{\uw} \in \uDiagBS(\fh,W)$, in the same way as for $\oB_w$, resp.~$\oB_{\uw}$.


\begin{rmk}
\phantomsection
\label{rmk:DBS}
\begin{enumerate}
\item
\label{it:rmk-DBS-dual}
Of course, all the constructions above can also be considered for the realization $\fh^*$ of~\S\ref{ss:dual-real}, giving rise to the category $\DiagBS(\fh^*,W)$ and all its cousins. To distinguish the two cases, the object of $\DiagBS(\fh^*,W)$ attached to an expression $\uw$ will be denoted $B^\wedge_\uw$. Similar conventions will be used for the objects $B_w$, $\oB_w$, $\uB_w$.
\item
\label{it:rmk-symmetry-DBS}
It is a standard fact that the category $\DiagBS(\fh,W)$ admits a canonical autoequivalence induced by reflecting diagrams along a vertical axis. This autoequivalence exchanges left and right multiplication by polynomials, hence induces an equivalence between $\uDiagBS(\fh,W)$ and $\oDiagBS(\fh,W)$.
\end{enumerate}
\end{rmk}

\subsection{Abe's category}
\label{ss:Abe}

Below we will also use a different incarnation of the Hecke category attached to $(W,S)$ and $\fh$, which we will denote by $\AbeBS(\fh,W)$, and which was introduced by Abe in~\cite{abe1}.

\begin{rmk}
Although this is not written explicitly, the conventions on realizations in~\cite{abe1,abe2} are different from those of~\cite{ew,ew20} (which we follow here).
Namely, in~\cite{abe1,abe2} a realization is a triple $(V, (\alpha_s : s\in S), (\alpha_s^\vee : s\in S))$ where $\alpha_s\in V$ and $\alpha_s^\vee\in V^*$, and the algebra $R$ is defined as the symmetric algebra of $V$. In other words, the module ``$V$'' of~\cite{abe1,abe2} is the module ``$V^*$'' of~\cite{ew,ew20}. Here we have decided to follow the conventions of~\cite{ew,ew20}; we will therefore translate all the results and constructions from~\cite{abe1,abe2} into these conventions.
\end{rmk}

In order to construct the category $\AbeBS(\fh,W)$, Abe first introduces the category $\cC(\fh,W)$ (denoted $\mathcal{C}'$ in~\cite{abe1,abe2}\footnote{In fact the definition of $\mathcal{C}'$ in~\cite{abe1} is slightly different, but this creates difficulties. The definition we use below solves these problems, as explained in~\cite{abe2}.}) whose objects are the triples
\[
(M,(M_{Q}^w)_{w\in W},\xi_M)
\]
where $M$ is a graded $R$-bimodule, each $M_{Q}^w$ is a graded $(R,Q)$-bimodule such that $m \cdot f=w(f) \cdot m$ for any $m\in M_{Q}^w$ and $f\in R$, these bimodules being $0$ except for finitely many $w$'s, and
\begin{equation}
\label{eqn:xi-Abe}
\xi_M:M\otimes_R Q \to \bigoplus_{w\in W}M_Q^w 
\end{equation}
is an isomorphism of graded $(R,Q)$-bimodules. A morphism in $\cC(\fh,W)$ from the object $(M,(M_{Q}^w)_{w\in W},\xi_M)$ to $(N,(N_{Q}^w)_{w\in W},\xi_N)$ is a morphism of graded $R$-bimodules $\varphi:M \to N$  such that 
\[
\left(\xi_N\circ(\varphi\otimes_{R}Q)\circ\xi_M^{-1}\right)(M_{Q}^w)\subset N_{Q}^w 
\]
for any $w\in W$. This category has a natural monoidal structure induced by $\otimes_{R}$, with neutral object the $R$-bimodule $R$ (upgraded in the obvious way to an object of $\cC(\fh,W)$). The shift-of-grading functor $\langle 1 \rangle$ (see~\S\ref{ss:gradings}) induces in the natural way an autoequivalence of $\cC(\fh,W)$, which will be denoted similarly. For simplicity, we often write $M$ for $(M,(M_{Q}^w)_{w\in W},\xi_M)$.

For $s\in S$, let
\[
R^s=\{f\in R\mid s \cdot f=f\}
\]
be the subring of $s$-invariant elements in $R$,
and choose an element $\delta_s\in V^*$ such that $\langle\delta_s,\alpha_s^\vee\rangle=1$. (Such a vector exists because our realization is assumed to satisfy Demazure surjectivity.) The $R$-bimodule $B_s=R \otimes_{R^s} R \langle -1 \rangle$ can be upgraded to an object in $\cC(\fh,W)$ by setting
\begin{equation}
\label{eq:Bs in Abe}
(B_s)^e_{Q}=Q(\delta_s\otimes 1-1\otimes s(\delta_s)), \quad
(B_s)^s_{Q}=Q(\delta_s\otimes 1-1\otimes\delta_s)
\end{equation}
and $(B_s)^w_{Q}=0$ for all $w\notin\{e,s\}$; see~\cite[\S2.4]{abe1} or~\cite[Chap.~II, \S 3.1.4]{curso}.

The category $\AbeBS(\fh,W)$ is defined as the smallest full subcategory of $\cC(\fh,W)$ which contains the neutral object $R$ and the objects $(B_s : s \in S)$ and is stable under the monoidal product $\otimes_R$ and the shift functor $\langle 1 \rangle$. In other words, the objects in $\AbeBS(\fh,W)$ are the objects of the form
\[
B_{s_1}\otimes_R\cdots\otimes_R B_{s_r} \langle n \rangle
\]
with $r \in \Z_{\geq 0}$, $s_1, \cdots, s_r \in S$ and $n\in\Z$. As in the diagrammatic category we will set
\[
\Hom^\bullet_{\AbeBS(\fh,W)}(X,Y) = 
\bigoplus_{n \in \Z} \Hom_{\AbeBS(\fh,W)}(X,Y \langle -n \rangle),
\]
considered as a graded vector space with $\Hom_{\AbeBS(\fh,W)}(X,Y \langle -n \rangle)$ in degree $n$.
We will denote by $\AbeBS^\oplus(\fh,W)$ the additive hull of $\AbeBS(\fh,W)$, and by $\Abe(\fh,W)$ the karoubian envelope of $\AbeBS^\oplus(\fh,W)$.

For an expression $\uw=(s_1, \cdots, s_n)$, we define the object $B_{\uw}$ in $\AbeBS(\fh,W)$ as
\[
B_{\uw}:=B_{s_1}\otimes_R\cdots\otimes_R B_{s_n}=R\otimes_{R^{s_1}}\cdots \otimes_{R^{s_n}}R \langle -n \rangle
\]
if $n \geq 1$, and $B_\varnothing=R$. The so-called $1$-tensor element
\[
u_{\uw}=(1\otimes 1)\otimes_R\cdots\otimes_R(1\otimes 1)\in B_{\uw}
\]
plays a singular role in the theory. (In case $\uw=\varnothing$ is the empty expression, the element $u_\varnothing$ is interpreted as $1 \in R$.)

As the reader might have noticed, we have used the same notation as for some objects in $\DiagBS(\fh,W)$. This should not lead to any confusion, because of the following result due to Abe (see~\cite[Theorem 3.15]{abe2}).

\begin{theorem}
\label{thm:Abe-EW}
There exists an equivalence of monoidal categories
\[
\DiagBS(\fh,W)\simto\AbeBS(\fh,W)
\]
which intertwines the autoequivalences $(1)$ and $\langle -1 \rangle$ and
sends $B_\uw$ to $B_\uw$ for any $\uw \in \Exp(W)$. 
\end{theorem}

\begin{rmk}
The proof of Theorem~\ref{thm:Abe-EW} has two parts: (a) the construction of the functor, and (b) the proof that it is an equivalence. Once (a) has been solved in the appropriate way, (b) is guaranteed by the results of~\cite{abe1}. The current proof of (a) relies on the computations in~\cite{abe2}.
We believe that considerations similar (or, in a sense, ``Koszul dual'') to those in Sections~\ref{sec:functor V}--\ref{sec:construction} (using, for $W$ finite, the indecomposable object $B_{w_0}$ rather than $\cT_{w_0}$) can be used to provide an alternative construction of this functor.
Details will appear elsewhere if this finds any application.
\end{rmk}

\section{Perverse and tilting sheaves}\label{sec:tilting}

In this section we briefly recall the definitions of a series of homotopy-type categories constructed from the Hecke category. These categories are the main objects of study of~\cite{amrw1} and~\cite{arv}. We also extend some results of~\cite[Chap.~10--11]{amrw1} to our present setting. 

\subsection{The biequivariant category}
\label{ss:BE}

The biequivariant category $\BE(\fh,W)$ is defined in \cite[\S4.2]{amrw1} as
\[
\BE(\fh,W) := \Kb \DiagBS^\oplus(\fh,W);
\]
the natural functor from $\BE(\fh,W)$ to $\Kb \Diag(\fh,W)$ is an equivalence, see~\cite[Lemma 4.9.1]{amrw1}, and we will therefore identify these categories whenever convenient.
The biequivariant category is monoidal for a certain product $\ustar$ which restricted to $\DiagBS(\fh,W)$ coincides with $\star$ and is triangulated on both sides; see~\cite[\S 4.2]{amrw1} for details.

The cohomological shift functors on the triangulated category $\BE(\fh,W)$ is denoted $[1]$. The shift-of-grading functor $(1)$ on $\BE(\fh,W)$ is the functor sending a complex $(\cF^n, d^n)_{n \in \Z}$ to the complex $(\cF^n(1), -d^n)_{n \in \Z}$. We also have the shift functor $\langle1\rangle=(-1)[1]$. 

As in \cite[\S 4.2]{amrw1}, given $\cF,\cG$ in $\BE(\fh,W)$, we will denote by 
\[
\HHom_{\BE(\fh,W)}(\cF,\cG)
\]
the bigraded $\bk$-vector space whose homogeneous components are
\[
\HHom_{\BE(\fh,W)}(\cF,\cG)_j^i:=\Hom_{\BE(\fh,W)}(\cF,\cG[i]\langle-j\rangle)
\]
for all $i,j\in\Z$. We also set $\EEnd_{\BE(\fh,W)}(\cF)=\HHom_{\BE(\fh,W)}(\cF,\cF)$.

Following~\cite[Example~4.2.2]{amrw1}, we define the standard object $\Delta_\uw$ and the costandard object $\nabla_\uw$ for any expression $\uw=(s_1, \dots, s_n)$ as
\[
\Delta_\uw=\Delta_{s_1}\ustar\cdots\ustar\Delta_{s_n}\quad\mbox{and}\quad\nabla_\uw=\nabla_{s_1}\ustar\cdots\ustar\nabla_{s_n},
\]
where $\Delta_s$ and $\nabla_s$ denote the complexes
\begin{equation}
\label{eq:Delta s}
 \cdots 0 \to B_s \xrightarrow{\begin{tikzpicture}[thick,scale=0.07,baseline]
      \draw (0,-5) to (0,0);
      \node at (0,0) {$\bullet$};
    \end{tikzpicture}} B_\varnothing(1) \to 0 \cdots
    \quad\mbox{and}\quad
    \cdots 0 \to B_\varnothing(-1) \xrightarrow{\begin{tikzpicture}[thick,baseline,xscale=0.07,yscale=-0.07]
      \draw (0,-5) to (0,0);
      \node at (0,0) {$\bullet$};
    \end{tikzpicture}} B_s \to 0 \cdots
\end{equation}
concentrated in degrees $0$ and $1$, and $-1$ and $0$, respectively. (By convention, $\Delta_\varnothing=\nabla_\varnothing=B_\varnothing$.)

On the other hand, standard and costandard objects $\Delta_w$ and $\nabla_w$ in $\BE(\fh,W)$ were defined for every $w\in W$ in~\cite[\S6.3]{arv}.\footnote{In~\cite[\S6.3]{arv} such objects are defined for certain subsets $I \subset W$ containing $w$. Here we take $I=W$, and omit it from the notation, following the conventions in~\cite{arv}. The same comment applies to various constructions from~\cite{arv} considered below.} We have
\begin{equation}
\label{eq:ARV Prop 6.11}
\Delta_w\simeq\Delta_\uw\quad\mbox{and}\quad\nabla_w\simeq\nabla_\uw
\end{equation}
if $\uw$ is a reduced expression for $w$, see~\cite[Proposition 6.11]{arv}.

A t-structure on $\BE(\fh,W)$ is constructed in~\cite[\S7.2]{arv}; its heart will be denoted $\fP_{\BE}(\fh,W)$. It turns out that the (co)standard objects $\Delta_w$ and $\nabla_w$ ($w \in W$) belong to $\fP_{\BE}(\fh,W)$, see~\cite[Proposition~7.8]{arv}, and that the shift functor $\langle1\rangle$ is t-exact, see~\cite[Lemma 7.3]{arv}. For every $w\in W$, there is (up to scalar) a unique nonzero morphism $f_w:\Delta_w\to\nabla_w$ \cite[Lemma 6.6]{arv}; we let
\[
\rL_w:=\operatorname{Im} (f_w)
\]
be the image of this morphism in $\fP_{\BE}(\fh,W)$. Then the assignment $(w,n) \mapsto \rL_w\langle n\rangle$ induces a bijection between $W\times\Z$ and the set of isomorphism classes of simple objects in $\fP_{\BE}(\fh,W)$, see~\cite[\S8.1]{arv}.

The following statement, which follows from the construction of standard and costandard objects via the recollement formalism of~\cite[\S 5]{arv}, will be required below.

\begin{lema}
\label{lem:cone-D-N}
 For any $w \in W$, the cone of any nonzero morphism $\Delta_w \to \nabla_w$ in $\BE(\fh,W)$ belongs to the triangulated subcategory generated by the objects of the form $\Delta_v(n)$ with $v \in W$ satisfying $v<w$ and $n \in \Z$.
\end{lema}

\subsection{The right-equivariant category}
\label{ss:RE}

The right-equivariant category is defined in \cite[\S4.3]{amrw1} as
\[
\RE(\fh,W) := \Kb \oDiagBS^\oplus(\fh,W);
\]
the natural functor from $\RE(\fh,W)$ to
$\Kb \oDiag(\fh,W)$ is an equivalence, see~\cite[Lemma 4.9.1]{amrw1}, and we will therefore identify these two categories whenever convenient.
The category $\RE(\fh,W)$ is in a natural way a right module category for the monoidal category $(\BE(\fh,W),\ustar)$; the action bifunctor will also be denoted $\ustar$, it is triangulated on both sides. There exists a natural ``forgetful'' functor
\[
\For^\BE_\RE : \BE(\fh,W) \to \RE(\fh,W)
\]
induced by~\eqref{eqn:Diag-oDiag}; this functor satisfies
\begin{equation}
\label{eqn:For-convolution}
\For^\BE_\RE (\cF \ustar \cG) = \For^\BE_\RE(\cF) \ustar \cG
\end{equation}
for $\cF,\cG$ in $\BE(\fh,W)$. The shifts functors $[1]$, $(1)$ and $\langle1\rangle$, and the bigraded vector spaces $\HHom_{\RE(\fh,W)}(\cF,\cG)$ are defined as for the biequivariant category. Then the functor $\For^\BE_\RE$ commutes with the functors $[1]$, $(1)$ and $\langle1\rangle$ in the obvious way.

A t-structure on $\RE(\fh,W)$ is also constructed in \cite[\S9]{arv}; its heart will be denoted $\fP_{\RE}(\fh,W)$. For this structure, the functors $\langle1\rangle$ and $\For^\BE_\RE$ are t-exact. 
The category $\fP_{\RE}(\fh,W)$ shares many properties with the categories of Bruhat-constructible perverse sheaves on flag varieties of Kac--Moody groups, but ``with an extra grading;'' in particular it has a natural structure of graded highest weight category (in the sense of~\cite[Appendix]{ar}\footnote{Compared to this reference, we make two modifications. First we use the term ``highest weight'' instead of ``quasihereditary.'' Second, we allow our weight poset $\mathcal{S}$ to be infinite, but with the condition that for any $s \in \mathcal{S}$ the set $\{t \in \mathcal{S} \mid t \leq s\}$ is finite; this does not affect the theory, except for the existence of enough projective objects, which is not used here. See e.g.~\cite[\S A]{curso} for the ungraded setting.}) with weight poset $(W,\leq)$ and normalized standard and costandard objects
\[
\oDelta_w:=\For^\BE_\RE(\Delta_w)\quad\mbox{and}\quad\onabla_w=\For^\BE_\RE(\nabla_w)
\]
for $w\in W$, see~\cite[Theorem~9.6]{arv}. (In case $W$ is a the Weyl group of a reductive group, and for an appropriate choice of $\fh$, one can in fact relate explicitly $\fP_{\RE}(\fh,W)$ with the corresponding category of perverse sheaves; see~\cite{ar} for details.)
The restriction of the functor $\For^\BE_\RE$ to the hearts of the perverse t-structures defines a fully faithful functor $\fP_{\BE}(\fh,W) \to \fP_{\RE}(\fh,W)$, see~\cite[Proposition 9.4]{arv}. 
Up to isomorphism, the simple objects in $\fP_{\RE}(\fh,W)$ are the objects
\[
\overline{\rL}_w\langle n\rangle:=\For^\BE_\RE(\rL_w\langle n\rangle)
\]
for $(w,n)\in W\times\Z$.

\subsection{Tilting objects}
\label{ss:tilt}

Since $\fP_{\RE}(\fh,W)$ has a natural structure of graded highest weight category, it makes sense to consider its tilting objects,
i.e.~the objects admitting both a filtration with subquotients 
of the form $\oDelta_w \langle n \rangle$ ($w \in W$, $n \in \Z$)
and a filtration with subquotients 
of the form $\onabla_w \langle n \rangle$ ($w \in W$, $n \in \Z$).
The full subcategory of $\fP_{\RE}(\fh,W)$ whose objects are tilting will be denoted $\Tilt_{\RE}(\fh,W)$. 
The category $\Tilt_{\RE}(\fh,W)$ is Krull--Schmidt, and its isomorphism classes of indecomposable objects is in a natural bijection with $W \times \Z$; for $w \in W$ we will denote by $\oT_w$ the indecomposable object corresponding to $(w,0)$. Then, for any $n \in \Z$, the indecomposable object corresponding to $(w,n)$ is
$\oT_w\langle n\rangle$. (For all of this, see~\cite[Appendix]{ar}.)
It follows from~\cite[Lemma~A.5 and Lemma~A.6]{ar} that the natural functors
\begin{equation}
\label{eq: KTiltRE DPRE RE}
\Kb \Tilt_{\RE}(\fh,W) \to \Db \fP_{\RE}(\fh,W) \to \RE(\fh,W) 
\end{equation}
are equivalences of triangulated categories.

\begin{example}
\label{ex:simple 1}
The simple object in $\fP_\RE(\fh,W)$ corresponding to  the neutral element of $W$ is $\oB_\varnothing$, viewed as a complex concentrated in degree $0$, and it coincides with the object $\oD_1=\onabla_1=\oT_1$.
\end{example}

\begin{example}
Let $s\in S$. The indecomposable tilting object $\oT_s$ in $\Tilt_{\RE}(\fh,W)$ is the complex 
\[
 \cdots 0 \to \oB_\varnothing(-1) \xrightarrow{\begin{tikzpicture}[thick,baseline,xscale=0.07,yscale=-0.07]
      \draw (0,-5) to (0,0);
      \node at (0,0) {$\bullet$};
    \end{tikzpicture}}
 \oB_s \xrightarrow{\begin{tikzpicture}[thick,scale=0.07,baseline]
      \draw (0,-5) to (0,0);
      \node at (0,0) {$\bullet$};
    \end{tikzpicture}} \oB_\varnothing(1) \to 0 \cdots
\]
concentrated in degrees $-1$, $0$ and $1$. 
It fits in exact sequences
\[
0\to\oD_s\to\oT_s\to\oD_1\langle1\rangle\to0\quad\text{and}\quad
0\to\onabla_1\langle-1\rangle\to\oT_s\to\onabla_s\to0.
\]
For this, see~\cite[Example 4.3.4]{amrw1}.
\end{example}

\subsection{The left-equivariant category}
\label{ss:LE}

We now define the left-equivariant category $\LE(\fh,W)$ by
\[
\LE(\fh,W):=\Kb \uDiagBS^\oplus(\fh,W).
\]
Of course, all the constructions and properties of $\RE(\fh,W)$ have analogues for $\LE(\fh,W)$. (In fact, these two categories are equivalent by Remark~\ref{rmk:DBS}\eqref{it:rmk-symmetry-DBS}.) In particular, we will denote by $\For^\BE_{\LE} : \BE(\fh,W) \to \LE(\fh,W)$ the natural ``forgetful" functor (induced by the natural functor $\DiagBS^\oplus(\fh,W) \to \uDiagBS^\oplus(\fh,W)$), and by
\[
\uDelta_w:=\For^\BE_{\LE}(\Delta_w),\quad\text{resp.}\quad\unabla_w=\For^\BE_{\LE}(\nabla_w)
\]
the standard, resp.~costandard, object associated with $w\in W$. The category $\LE(\fh,W)$ admits a natural ``perverse'' t-structure whose heart, denoted $\fP_{\LE}(\fh,W)$, admits a canonical structure of graded highest weight category with weight poset $(W, \leq)$ and normalized standard, resp.~costandard, objects the objects $(\uDelta_w : w \in W)$, resp.~$(\unabla_w : w \in W)$. Its full subcategory of tilting objects will be denoted $\Tilt_{\LE}(\fh,W)$. The indecomposable objects in $\Tilt_{\LE}(\fh,W)$ are in a canonical bijection with $W \times \Z$, and the indecomposable object associated with $(w,0)$ will be 
denoted $\uT_w$ (for $w\in W$).

Recall that $\uDiagBS^\oplus(\fh,W)$ is a left module category for the monoidal category $\DiagBS^\oplus(\fh,W)$. This structure ``extends'' to a bifunctor
\[
 \ustar : \BE(\fh,W) \times \LE(\fh,W) \to \LE(\fh,W)
\]
which defines a left action of the monoidal category $(\BE(\fh,W),\ustar)$ on $\LE(\fh,W)$. This bifunctor is triangulated on both sides, and for $\cF,\cG \in \BE(\fh,W)$ we have a canonical (in particular, bifunctorial) isomorphism
\[
 \For^{\BE}_{\LE}(\cF \ustar \cG) \cong \cF \ustar \For^{\BE}_{\LE}(\cG).
\]

\begin{rmk}
 As in Remark~\ref{rmk:DBS}\eqref{it:rmk-DBS-dual}, below we will also consider all the constructions above for the realization $\fh^*$ instead of $\fh$. We will add a superscript ``$\wedge$'' to all our notations in this case.
\end{rmk}

\subsection{Free-monodromic category}
\label{ss:FM}

We will denote by
\[
\FM(\fh,W)
\]
the ``free-monodromic'' category defined in~\cite[\S5.1]{amrw1}.
It is a $\bk$-linear additive category, whose objects are sequences of objects in $\DiagBS^\oplus(\fh,W)$ endowed with a kind of differential; the precise construction is rather technical, and will not be recalled here. The category $\FM(\fh,W)$ has shift functors $[1]$, $(1)$ and $\langle1\rangle$ commuting with each other and such that $\langle1\rangle=(-1)[1]$. (Note that $[1]$ is a priori not the suspension functor for a triangulated structure on $\FM(\fh,W)$; in fact it is not known at this point whether this category admits a triangulated structure.) As in the biequivariant category (see~\S\ref{ss:BE}), for $\cF,\cG \in \FM(\fh,W)$ we define the bigraded vector space $\HHom_{\FM(\fh,W)}(\cF,\cG)$ by setting
\[
\HHom_{\FM(\fh,W)}(\cF,\cG)_j^i:=\Hom_{\FM(\fh,W)}(\cF,\cG[i]\langle-j\rangle)
\]
and $\EEnd_{\FM(\fh,W)}(\cF) = \HHom_{\FM(\fh,W)}(\cF,\cF)$.
We point out that we have
\begin{equation}
\label{eq:HHom FM and shift}
\HHom_{\FM(\fh,W)}(\cF,\cG\langle1\rangle)=\HHom_{\FM(\fh,W)}(\cF,\cG) \langle 1 \rangle
\end{equation}
for all $\cF, \cG \in \FM(\fh,W)$.

As in~\S\ref{ss:gradings} we consider the graded ring $R^\vee$ as a bigraded ring with nonzero components concentrated in $\{0\} \times \Z$.
Then as explained in~\cite[\S5.2]{amrw1}, for $\cF, \cG \in \FM(\fh,W)$ the bigraded vector space
$\HHom_{\FM(\fh,W)}(\cF,\cG)$ has a natural structure of bigraded $R^\vee$-bimodule.
Both actions are denoted $\wstar$ and are compatible with composition in the sense that $x\wstar (f \circ g) = (x\wstar f) \circ g = f \circ (x\wstar g)$ for $x \in R^\vee$, and similarly for the right action. In particular, the left action is induced by certain bigraded algebra homomorphisms
\[
\mu_{\cF}:R^\vee\to\EEnd_{\FM(\fh,W)}(\cF)
\]
for any object $\cF$; if $f\in\HHom_{\FM(\fh,W)}(\cF,\cG)$ and $x\in R^\vee$, then $x\wstar f=\mu_\cG(x)\circ f=f\circ\mu_\cF(x)$.

It should be the case that $\FM(\fh,W)$ (or an appropriate subcategory) admits a structure of monoidal category. Unfortunately this construction turns out to be delicate, and no general answer is known at present. But at least part of this structure has been constructed in~\cite[Chap.~6--7]{amrw1}. Explicitly, there is a notion of ``convolutive complexes'' in $\FM(\fh,W)$, see~\cite[Definition 6.1.1]{amrw1}, and
the full subcategory of $\FM(\fh,W)$ whose objects are the convolutive complexes is equipped with an operation $\wstar$ such that $\cF\wstar(-)$ and $(-)\wstar\cF$ are functors for any fixed convolutive object $\cF$. Moreover, for a fixed convolutive complex $\cF$ the functor $\cF\wstar(-)$ has an extension to the whole category $\FM(\fh,W)$, see~\cite[Proposition 7.6.3]{amrw1}. The operation $\wstar$ satisfies all the axioms of a monoidal category except possibly for the ``interchange law'' stating that for $f : \cF \to \cG$, $g : \cG \to \cH$, $h : \cF' \to \cG'$ and $k : \cG' \to \cH'$ morphisms between convolutive complexes we have
\[
(g \circ f) \wstar (k \circ h)=
(g \wstar k) \circ (f \wstar h);
\]
see \cite[Chap.~7]{amrw1} for details. In particular:
\begin{itemize}
\item
 the operation $\wstar$ is associative, so that we can omit parentheses when considering multiple instances of $\wstar$, see~\cite[\S 7.2]{amrw1};
 \item
 we have a convolutive object $\wT_\varnothing$ constructed in~\cite[\S 5.3.1]{amrw1} and which behaves like a unit object; see~\cite[\S 7.1]{amrw1}.
 \end{itemize}
 We will return to this question in~\S\ref{ss:bifunctoriality} below.

\subsection{Left-monodromic category}
\label{ss:LM}

We will denote by
\[
\LM(\fh,W)
\]
the ``left monodromic'' category defined in~\cite[\S4.4]{amrw1}.
As for the free-mono\-dromic category, the objects in this category are sequences of objects in $\DiagBS^\oplus(\fh,W)$ endowed with (another kind of) ``differential''. 
There are shift functors $[1]$, $(1)$ and $\langle1\rangle$ with analogous properties to those in $\FM(\fh,W)$, which allows to define bigraded vector spaces $\HHom_{\LM(\fh,W)}(\cF,\cG)$ for $\cF,\cG \in \LM(\fh,W)$ by the same recipe as in $\FM(\fh,W)$. This time, the bigraded space $\HHom_{\LM(\fh,W)}(\cF,\cG)$ has a canonical structure of bigraded left $R^\vee$-module. There exists a functor
\[
\For^\FM_\LM:\FM(\fh,W)\to \LM(\fh,W)
\]
which satisfies
\[
\For^\FM_\LM \circ (1) = (1) \circ \For^\FM_\LM, \quad \For^\FM_\LM \circ [1] = [1] \circ \For^\FM_\LM, \quad \For^\FM_\LM \circ \langle 1 \rangle = \langle 1 \rangle \circ \For^\FM_\LM.
\]
Moreover, $\LM(\fh,W)$ has a natural structure of triangulated category with suspension functor $[1]$. There is also a natural right action of the monoidal category $(\BE(\fh,W),\ustar)$ on $\LM(\fh,W)$; the corresponding bifunctor will again be denoted $\ustar$. (See~\cite[(4.23)]{amrw1} for an explicit construction.)

There is a notion of convolutive objects in $\LM(\fh,W)$ and, for $\cF \in \FM(\fh,W)$ and $\cG \in \LM(\fh,W)$ both convolutive, a convolutive object $\cF \wstar \cG \in \LM(\fh,W)$. There is also a ``convolution'' operation on morphisms, such that for $\cF \in \FM(\fh,W)$ and $\cG \in \LM(\fh,W)$ convolutive the operations $\cF \wstar (-)$ and $(-) \wstar \cG$ are functorial; see~\cite[\S 6.6]{amrw1}.
Moreover, for a fixed convolutive object $\cF \in \FM(\fh,W)$ the functor $\cF \wstar (-)$ extends to a triangulated functor from $\LM(\fh,W)$ to itself, see~\cite[Proposition~7.6.4]{amrw1}.
The functor $\For^\FM_\LM$ sends convolutive complexes to convolutive complexes and satisfies
\begin{equation}\label{eq:wstar and For}
\For^\FM_\LM(\cF\wstar\cG)\simeq\cF\wstar \For^\FM_\LM(\cG)
\end{equation}
for all $\cF,\cG$ in $\FM(\fh,W)$ convolutive, see~\cite[$(6.18)$]{amrw1}. 

There is also an equivalence of triangulated categories
\begin{equation}
\label{eqn:ForLMRE}
\For^\LM_\RE:\LM(\fh,W) \simto \RE(\fh,W)
\end{equation}
given in \cite[Theorem 4.6.2]{amrw1} and we have a functor
\[
\For^\BE_\LM:\BE(\fh,W)\to \LM(\fh,W)
\]
such that $\For^\LM_\RE\circ \For^\BE_\LM=\For^\BE_\RE$, cf. \cite[\S 4.6]{amrw1}. By \cite[Lemma 6.6.1]{amrw1}, for any $\cG\in\BE(\fh,W)$ the object $\For^{\BE}_{\LM}(\cG)$ is convolutive, and for all $\cF\in\FM(\fh,W)$ we have a canonical isomorphism
\[
\cF\wstar \For^\BE_\LM(\cG)\simeq\For^\FM_\LM(\cF)\ustar\cG.
\]
The functors $\For^\LM_\RE$ and $\For^\BE_\LM$ commute with the shift functors in the obvious way.

We can use that $\For^\LM_\RE$ is an equivalence to translate all of the structures and properties of $\RE(\fh,W)$ to $\LM(\fh,W)$. Explicitly, we endow $\LM(\fh,W)$ with the t-structure obtained from that of $\RE(\fh,W)$ (see~\S\ref{ss:RE}) and denote by $\fP_\LM(\fh,W)$ its heart, i.e.~the inverse image of $\fP_\RE(\fh,W)$ under the equivalence $\For^\LM_\RE$. Of course, this category is stable under $\langle1\rangle$ and inherits the graded highest weight structure from $\fP_\RE(\fh,W)$. The normalized standard and costandard objects are
\begin{equation}
\label{eq:standard in LM}
\dDelta_w:=\For^\BE_\LM(\Delta_w)\quad\mbox{and}\quad\dnabla_w:=\For^\BE_\LM(\nabla_w)
\end{equation}
for all $w\in W$, since $\For^\LM_\RE(\For^\BE_\LM(\Delta_w))=\For^\BE_\RE(\Delta_w)=\oDelta_w$ and similarly for the costandard object. The objects 
\[
\cL_w\langle n\rangle=\For^\BE_\LM(\rL_w\langle n\rangle),
\]
$(w,n)\in W\times \Z$, form a family of representatives of the isomorphism classes of simple objects in $\fP_\LM(\fh,W)$. By Example \ref{ex:simple 1}, we have $\cL_1=\dDelta_1=\dnabla_1$.

The following property is useful to extend results from \cite{amrw2} to our more general context.

\begin{prop}\label{prop:L1 Delta w nabla w}
Let $w\in W$.
\begin{enumerate}
\item The socle of $\dDelta_w$ is isomorphic to $\cL_1\langle-\ell(w)\rangle$, and the cokernel of the inclusion $\cL_1\langle-\ell(w)\rangle\hookrightarrow\dDelta_w$ has no composition factor of the form $\cL_1\langle n\rangle$ with $n \in \Z$.
\item The head of $\dnabla_w$ is isomorphic to $\cL_1\langle\ell(w)\rangle$, and the kernel of the surjection $\dnabla_w\twoheadrightarrow\cL_1\langle\ell(w)\rangle$ has no composition factor of the form $\cL_1\langle n\rangle$ with $n \in \Z$.
\end{enumerate}
\end{prop}

\begin{proof}
The objects $\rL_1$, $\Delta_w$ and $\nabla_w$ satisfy an analogous statement in $\BE(\fh,W)$, see \cite[Proposition 8.3]{arv}. Since the functor $\For^{\BE}_{\RE} : \fP_\BE(\fh,W) \to \fP_\RE(\fh,W)$ is fully faithful, and since its essential image contains all simple objects, we deduce that the socle of $\dDelta_w$ is isomorphic to $\cL_1\langle-\ell(w)\rangle$ and the head of $\dnabla_w$ is isomorphic to $\cL_1\langle\ell(w)\rangle$. The other claims also follow, since $\For^\BE_\RE$ is exact and sends simple objects to simple objects.
\end{proof}

\subsection{Left-monodromic tilting category}\label{ss:LMT}

The following results were established for Cartan realizations of crystallographic Coxeter groups in~\cite[Chap.~10]{amrw1}. Using the theory developed in~\cite{arv}, we can extend them to our present setting, with identical proofs.

For $s \in S$, recall the object
$\wT_s \in \FM(\fh,W)$ defined in \cite[\S 5.3.2]{amrw1}; this object is convolutive by definition. We therefore have a triangulated functor $\wT_s\wstar(-):\LM(\fh,W)\to \LM(\fh,W)$, see~\S\ref{ss:LM}.

\begin{lema}
\label{le:tensoring by wT s}
Let $w\in W$ and $s\in S$. 
\begin{enumerate}
 \item 
 \label{it:tensor-Ts-1}
 If $sw>w$, we have distinguished triangles
 \[
\dDelta_{sw}\to \wT_s\wstar\dDelta_{w}\to \dDelta_{w}\langle1\rangle\xrightarrow{[1]}
 \quad\mbox{and}\quad\dnabla_{w}\langle-1\rangle\to\wT_s\wstar\dnabla_{w}\to \dnabla_{sw}\xrightarrow{[1]}
 \]
 in $\LM(\fh,W)$, where in each case the second morphism is nonzero.
\item 
\label{it:tensor-Ts-2}
If $sw<w$, we have distinguished triangles
\[
 \dDelta_{w}\langle-1\rangle\to \wT_s\wstar\dDelta_{w}\to \dDelta_{sw}\xrightarrow{[1]} \quad \mbox{and}\quad\dnabla_{sw}\to\wT_s\wstar\dnabla_{w}\to \dnabla_{w}\langle1\rangle\xrightarrow{[1]}
 \]
 in $\LM(\fh,W)$, where in each case the second morphism is nonzero.
\end{enumerate}
\end{lema}

\begin{proof}
The existence of the distinguished triangles can be obtained as in~\cite[Lemma 10.5.3]{amrw1}, citing~\cite[Proposition 6.11]{arv} instead of \cite[Proposition 4.4]{ar} when tensoring with (co)standard objects. We prove that the second morphism in the first triangle in~\eqref{it:tensor-Ts-1} is nonzero; the other cases are similar. Assume for a contradiction that this is not the case; then we have $\dDelta_{sw} \cong \wT_s\wstar\dDelta_{w} \oplus \dDelta_{w} \langle 1 \rangle [-1]$. This contradicts the fact that $\Hom(\dDelta_{sw}, \dDelta_{w} \langle 1 \rangle [-1])=0$, which follows e.g.~from the fact that both $\dDelta_{sw}$ and $\dDelta_{w} \langle 1 \rangle$ belong to $\fP_\LM(\fh,W)$.
\end{proof}

\begin{lema}
\label{le:tensoring by wT s is exact}
The triangulated functor $\wT_s\wstar(-):\LM(\fh,W)\to \LM(\fh,W)$ is t-exact with respect to the perverse t-structure. 
\end{lema}

\begin{proof}
By \cite[Lemma~7.5 and Remark~7.6]{arv}, the nonpositive and nonnegative parts of the t-structure are generated under extensions by appropriate shifts of standard and costandard objects, respectively. Thus, the claim follows from Lemma~\ref{le:tensoring by wT s}. 
\end{proof}

We set $\cT_1:=\For^{\BE}_{\LM}(B_\varnothing)$ and,
given an expression $\uw=(s_1, \dots, s_r)$, we set
\[
\cT_\uw:=\wT_{s_1}\wstar\cdots\wstar\wT_{s_r}\wstar\cT_1\in\LM(\fh,W).
\]
Let $\Tilt_\LM^\oplus(\fh,W)$ be the full subcategory of $\LM(\fh,W)$ whose objects are the direct sums of objects
of the form $\cT_{\uw}\langle n\rangle$ with $\uw \in \Exp(W)$ and $n \in \Z$. We define the left-monodromic tilting category $\Tilt_\LM(\fh,W)$ as its karoubian envelope. (Note that $\Tilt_\LM(\fh,W)$ is a full subcategory of $\LM(\fh,W)$ as the latter has a bounded t-structure and hence is karoubian by the main result of~\cite{lc}.) This definition is justified by the following result.

\begin{prop}
\label{prop:TiltLM}
The functor $\For^\LM_\RE$ induces an equivalence of additive categories
\[
\Tilt_\LM(\fh,W) \simto \Tilt_\RE(\fh,W).
\]
As a consequence, the natural functors
\begin{equation}
\label{eq: KTiltLM DPLM LM}
\Kb \Tilt_{\LM}(\fh,W) \to \Db \fP_{\LM}(\fh,W) \to \LM(\fh,W)
\end{equation}
are equivalences of triangulated categories.
\end{prop}

\begin{proof}
For the first claim, the proof of \cite[Proposition 10.5.1]{amrw1} applies in the present setting, using Lemma~\ref{le:tensoring by wT s} instead of~\cite[Lemma~10.5.3]{amrw1}. Then, the fact that the functors in~\eqref{eq: KTiltLM DPLM LM} are equivalences follows from the similar property for the functors in~\eqref{eq: KTiltRE DPRE RE}.
\end{proof}

\begin{rmk}
 Standard arguments show that the obvious functor
 \[
  \Kb \Tilt_{\LM}(\fh,W) \to \Kb \Tilt_{\LM}^\oplus(\fh,W)
 \]
is an equivalence of triangulated categories. Hence we also obtain an equivalence of categories
\begin{equation}
\label{eq:KTiltLM+-LM}
\Kb \Tilt_{\LM}^\oplus(\fh,W) \simto \LM(\fh,W),
\end{equation}
as in~\cite[Lemma~2.4]{amrw2}.
\end{rmk}

Using Proposition~\ref{prop:TiltLM} we can transfer to $\Tilt_\LM(\fh,W)$ the usual properties satisfied by the tilting objects of a graded highest weight category, cf.~\cite[\S 9.5]{arv}, and obtain the following statement.

\begin{cor}
\label{cor:Tilt LM indecomposable}
The category $\Tilt_\LM(\fh,W)$ is Krull--Schmidt. For any $w\in W$, there exists a unique (up to isomorphism) indecomposable object $\cT_w$ characterized by the following properties:
\begin{enumerate}
\item for any reduced expression $\uw$ expressing $w$, $\cT_w$ occurs as a direct summand of $\cT_{\uw}$ with multiplicity $1$;
\item $\cT_w$ does not occur as a direct summand of any object $\cT_\uv \langle n \rangle$ with $\uv$ an expression such that $\ell(\uv)<\ell(w)$ and $n \in \Z$.
\end{enumerate}
Moreover, the assignment $(w,n)\mapsto\cT_w\langle n\rangle$ induces a bijection bewteen $W\times\Z$ and the set of isomorphism classes of indecomposable objects in $\Tilt_\LM(\fh,W)$. \qed
\end{cor}

\subsection{Free-monodromic tilting category}
\label{ss:FMT}

Given an expression $\uw=(s_1, \dots, s_r)$, we set
\[
\wT_\uw=\wT_{s_1}\wstar\cdots\wstar\wT_{s_r}\in\FM(\fh,W).
\]
Note that by~\eqref{eq:wstar and For} we have
$\cT_\uw=\For^\FM_\LM(\wT_\uw)$.

Proposition~\ref{prop:TiltLM} implies that for any expressions $\uv,\uw$ and $i,j \in \Z$ we have
\[
\Hom_{\LM(\fh,W)}(\cT_\uv,\cT_\uw[i]\langle j\rangle)=0 \quad \text{unless $i=0$}
\]
(see~\cite[$(9.2)$]{arv}). 
As in the proof of \cite[Corollary~10.6.2]{amrw1}, the above equality and \cite[Lemma 5.2.3]{amrw1} imply the following statement.

\begin{prop}
\label{prop: Hom FM between tilting}
For any expressions $\uv$, $\uw$ and any $i,j\in\Z$, we have 
\[
\HHom_{\FM(\fh,W)}(\wT_\uv,\wT_\uw)^i_j=0\quad\mbox{unless}\quad i=0.
\]
Moreover
$\HHom_{\FM(\fh,W)}(\wT_\uv,\wT_\uw)^0_\bullet$ is graded free as a right $R^\vee$-module, and the morphism
\[
\HHom_{\FM(\fh,W)}(\wT_\uv,\wT_\uw)^0_\bullet\otimes_{R^\vee}\bk \to \HHom_{\LM(\fh,W)}(\cT_\uv,\cT_\uw)^0_\bullet 
\]
induced by the functor $\For^\FM_\LM$ is an isomorphism. \qed
\end{prop}

Below we will need to consider the karoubian envelope $\FM^{\Kar}(\fh,W)$ of the category $\FM(\fh,W)$. The shift functors $\langle m \rangle$ extend to autoequivalences of $\FM^{\Kar}(\fh,W)$. Given $\cF,\cG$ in $\FM^{\Kar}(\fh,W)$, we will also consider the bigraded $\bk$-vector space
\[
 \HHom_{\FM^{\Kar}(\fh,W)}(\cF,\cG) = \bigoplus_{n,m \in \Z} \Hom_{\FM^{\Kar}(\fh,W)}(\cF,\cG \langle -m \rangle [n]).
\]
Using the fact that the left and right actions of $R^\vee$ on morphism spaces in $\FM(\fh,W)$ is compatible with composition, one sees that this space has a canonical structure of bigraded $R^\vee$-bimodule. The functor $\For^\FM_\LM : \FM(\fh,W) \to \LM(\fh,W)$ extends to $\FM^{\Kar}(\fh,W)$ by the universal property of the karoubian envelope. By a minor abuse of notation we also denote this extension by $\For^\FM_\LM$.

We define the free-monodromic tilting category $\TBS(\fh,W)$ as the full subcategory of $\FM(\fh,W)$ whose objects are those of the form $\wT_\uw\langle n\rangle$ with $\uw \in \Exp(W)$ and $n \in \Z$. 
We denote by $\TBS^\oplus(\fh,W)$ the additive hull of $\TBS(\fh,W)$, and by $\rT(\fh,W)$ the karoubian envelope of $\TBS^\oplus(\fh,W)$ (a full subcategory in $\FM^{\Kar}(\fh,W)$).
By construction, the convolution operation $\wstar$ restricts to an operation
\begin{equation}
\label{eqn:convolution-tilt-FM-LM}
\wstar : \TBS(\fh,W) \times \Tilt_\LM(\fh,W) \to \Tilt_\LM(\fh,W).
\end{equation}

Proposition~\ref{prop: Hom FM between tilting} extends to all objects in $\rT(\fh,W)$, as follows.

\begin{prop}
\label{prop: analogous of Hom FM between tilting}
Let $\cF,\cG$ be objects in $\rT(\fh,W)$. Then we have
\[
\HHom_{\FM^{\Kar}(\fh,W)}(\cF,\cG)^i_j=0 \quad \text{unless} \quad i=0,
\]
$\HHom_{\FM^{\Kar}(\fh,W)}(\cF,\cG)$ 
is graded free as a right $R^\vee$-module,
and the morphism
\[
\HHom_{\FM^{\Kar}(\fh,W)}(\cF,\cG)
\otimes_{R^\vee}\bk\to\HHom_{\LM(\fh,W)}(\For^\FM_\LM(\cF),\For^\FM_\LM(\cG))
\]
induced by the functor $\For^\FM_\LM$ is an isomorphism. 
\end{prop}

\begin{proof}
By construction of the karoubian envelope, 
$\HHom_{\FM^{\Kar}(\fh,W)}(\cF,\cG)$
identifies with a direct summand of a similar space of morphisms in $\FM(\fh,W)$. Hence this space vanishes in degrees in $(\Z \smallsetminus \{0\}) \times \Z$ and is a projective graded $R^\vee$-module by Proposition~\ref{prop: Hom FM between tilting} and thus is graded-free.
The fact that $\For^\FM_\LM$ induces an isomorphism also follows from Proposition~\ref{prop: Hom FM between tilting}.
\end{proof}

The classification of the indecomposable tilting objects in $\Tilt_\LM(\fh,W)$ can be ``upgraded'' to $\rT(\fh,W)$ using Proposition~\ref{prop: analogous of Hom FM between tilting}, as in \cite[Theorem 10.7.1]{amrw1}.

\begin{theorem}
\label{thm:tilting FM}
The category $\rT(\fh,W)$ is Krull--Schmidt. For any $w\in W$, there exists a unique (up to isomorphism) indecomposable object $\wT_w$ such that
\[
\For^\FM_\LM(\wT_w)=\cT_w.
\]
In addition, $\wT_w$ is characterized by the following properties:
\begin{enumerate}
\item for any reduced expression $\uw$ expressing $w$, $\wT_w$ occurs as a direct summand of $\wT_{\uw}$ with multiplicity $1$;
\item $\wT_w$ does not occur as a direct summand of any object $\wT_\uv$ with $\uv$ an expression such that $\ell(\uv)<\ell(w)$ and $n \in \Z$.
\end{enumerate}
Moreover, the assignment $(w,n)\mapsto\wT_w\langle n\rangle$ induces a bijection bewteen $W\times\Z$ and the set of isomorphism classes of indecomposable objects in $\rT(\fh,W)$. \qed
\end{theorem}

In case $w=1$ we have $\wT_1=\wT_\varnothing$, and if $\uw=s \in S$ then $\wT_s$ is the object denoted in this way above. In these cases the object $\wT_w$ is canonical; in general, it is only defined up to isomorphism.

\subsection{The bifunctoriality assumption}
\label{ss:bifunctoriality}

As explained in~\S\ref{ss:FM}, it is not known whether the operation $\wstar$ is a bifunctor on the subcategory of convolutive objects in $\FM(\fh,W)$. As explained in~\cite[\S 7.7]{amrw1}, it is expected that this condition holds at least on the subcategory $\TBS(\fh,W)$, which boils down to the fact that for $f : \cF \to \cG$, $g : \cG \to \cH$, $h : \cF' \to \cG'$ and $k : \cG' \to \cH'$ morphisms between objects in $\TBS(\fh,W)$ we have
\begin{equation}
\label{eqn:exchange}
(g \circ f) \wstar (k \circ h)=
(g \wstar k) \circ (f \wstar h).
\end{equation}
Below we will assume that this property holds for our given realization, which implies that $\wstar$ induces a monoidal structure on the category $\TBS(\fh,W)$ (and then on $\TBS^\oplus(\fh,W)$ and on $\rT(\fh,W)$).

\begin{rmk}
 \phantomsection
 \label{rmk:bifunctoriality}
\begin{enumerate}
\item 
In case our realization is a Cartan realization of a crystallographic Coxeter group, the main result of~\cite[Chap.~11]{amrw1} states that the condition above is satisfied.
\item
Using the results of~\S\ref{ss:FMT}, the methods of~\cite[Chap.~11]{amrw1} also apply for a general realization of a general Coxeter group, provided for any pair $s,t \in S$ of distinct simple reflections generating a finite subgroup of $W$, the conditions in~\cite[Chap.~8]{amrw1} are satisfied by $\fh_{|\langle s,t \rangle}$.
This case covers at least, for any Coxeter system, the geometric realization and the Soergel realizations (see~\S\ref{ss:convention}), provided the condition~\eqref{it:assumptions-3} of~\S\ref{ss:convention} is satisfied in case $(W,S)$ admits a parabolic subsystem of type $\mathsf{H}_3$.
\item
Matthew Hogancamp and Shotaro Makisumi have announced a proof of this condition in full generality (provided the assumptions~\eqref{it:assumptions-1}--\eqref{it:assumptions-3}--\eqref{it:assumptions-4} of~\S\ref{ss:convention} are satisfied). Unfortunately, no written account of their proof is available at this point.
\item
\label{it-action-Tilt}
Using Proposition~\ref{prop: Hom FM between tilting} one obtains that, if the condition above is satisfied, the operation~\eqref{eqn:convolution-tilt-FM-LM} is a bifunctor, and defines on $\Tilt_\LM(\fh,W)$ a structure of module
category for the monoidal category $(\rT(\fh,W),\wstar)$.
\end{enumerate}
\end{rmk}

%

\section{A functor \texorpdfstring{$\V$}{V}}
\label{sec:functor V}

In this section we assume that the conditions of~\S\ref{ss:convention},~\S\ref{ss:dual-real} and~\S\ref{ss:bifunctoriality} are satisfied. In addition, we assume that $W$ is finite.

\subsection{Statement}
\label{ss:V-statement}

Recall the dual realization $\fh^*$ of $(W,S)$ considered in~\S\ref{ss:dual-real}. Our aim is to prove the following statement.

\begin{theorem}
\label{thm:V}
There exists a canonical equivalence of monoidal additive categories
\[
\V:\TBS(\fh,W) \simto \AbeBS(\fh^*,W) 
\]
which satisfies
$\V\circ\langle1\rangle= \langle -1 \rangle\circ\V$ and  $\V(\wT_\uw)\simeq B_\uw^\wedge$ for any expression $\uw$. 
\end{theorem}

We construct the functor $\V$ and prove the theorem in \S\ref{ss:V} and \S\ref{ss:Invertibility of V}. Before that we need some preliminaries.

\subsection{The big tilting object}
\label{ss: big tilting}

Let $w_0$ be the longest element in $W$, and consider the corresponding indecomposable tilting object $\cT_{w_0}$ in $\Tilt_\LM(\fh,W)$, see Corollary~\ref{cor:Tilt LM indecomposable}. We set 
\[
\cP:=\cT_{w_0}\langle-\ell(w_0)\rangle.
\]
By \cite[Theorem 10.3]{arv}, $\cP$ is the projective cover of the simple object $\cT_1=\cL_1$, and \cite[$(10.5)$]{arv} implies that
\begin{equation}
\label{eq:dim Hom LM P nabla}
\dim_\bk\Hom_{\LM(\fh,W)}(\cP,\dnabla_w\langle m\rangle)=
\begin{cases}
1&\mbox{if $m=-\ell(w)$;}\\
0&\mbox{otherwise.} 
\end{cases}
\end{equation}
In particular, we have
\begin{equation}
\label{eqn:Hom-P-T1}
\HHom_{\LM(\fh,W)}(\cP,\cT_1)^i_j=
\begin{cases}
\bk&\text{if $i=j=0$;}\\
0&\text{otherwise.} 
\end{cases}
\end{equation}

We fix an indecomposable tilting object $\wT_{w_0}$ in $\rT(\fh,W)$ such that $\For^\FM_\LM(\wT_{w_0})=\cT_{w_0}$, see Theorem \ref{thm:tilting FM}, and set 
\begin{equation}
\label{eqn:wP}
\wP:=\wT_{w_0}\langle-\ell(w_0)\rangle.
\end{equation}
Proposition \ref{prop: analogous of Hom FM between tilting} and~\eqref{eqn:Hom-P-T1} imply that we have an isomorphism of graded $R^\vee$-modules
\[
\HHom_{\FM^{\Kar}(\fh,W)}(\wP,\wT_1)
\simeq R^\vee, 
\]
hence in particular that
\[
\dim_\bk(\Hom_{\rT(\fh,W)}(\wP,\wT_1))=1.
\]
We fix from now on a nonzero morphism
\begin{equation}
\label{eq:xi}
\xi:\wP \to \wT_1,
\end{equation}
which is automatically a generator of 
$\HHom_{\FM^{\Kar}(\fh,W)}(\wP,\wT_1)$
as a right $R^\vee$-module. We also set $\xi'=\For^\FM_\LM(\xi)$, a generator of $\Hom_{\LM(\fh,W)}(\cP,\cT_1)$.

The objects $\wP$ and $\cP$ are studied in \cite[\S\S3.1--3.4]{amrw2} for Cartan realizations of (finite) crystallographic Coxeter groups. As in \S\ref{ss:FMT}, these results hold in our present setting, and their proofs can be copied from~\cite{amrw2}, replacing the references to~\cite{ar} by references to~\cite{arv}. Below we state the results we will use, and give sketches of proofs.


\begin{lema}
In the abelian category $\fP_\LM(\fh,W)$ we have
\[
[\cP:\cL_1\langle m\rangle]=0
\]
unless $m\leq0$, and moreover $[\cP:\cL_1]=1$. In particular, $\End_{\LM(\fh,W)}(\cP)=\bk\cdot\id$.
\end{lema}

\begin{proof}
The proof is similar to that of~\cite[Lemma 3.1]{amrw2}. Namely,
as $\cP$ is tilting it admits a standard filtration. Using the reciprocity formula (see e.g.~\cite[$(10.1)$]{arv}), for $w \in W$ and $n \in \Z$ we have 
\[
(\cP:\dDelta_w\langle n\rangle)=[\dnabla_w\langle n\rangle:\cL_1],
\]
which is equal to $1$ if $n=-\ell(w)$ and to $0$ otherwise by Proposition \ref{prop:L1 Delta w nabla w}. Using again Proposition \ref{prop:L1 Delta w nabla w} we deduce that
\[
[\cP:\cL_1\langle n\rangle]=\#\{w\in W\mid n=-2\ell(w)\},
\]
which implies the statement.
\end{proof}

Let $s\in S$ and $\wepsilon_s:\wT_s\to\wT_1\langle1\rangle$ be the morphism defined in \cite[\S5.3.4]{amrw1}. As in~\cite[\S3.3]{amrw2}, there exists a unique morphism
\[
\zeta_s':\cP\to\cT_s\langle-1\rangle  
\]
such that $\For^\FM_\LM(\wepsilon_s\langle-1\rangle)\circ\zeta_s'=\xi'$. In turn, as in~\cite[Lemma 3.5]{amrw2}, there exists a unique morphism
\begin{equation}
\label{eq:zeta s}
\zeta_s:\wP\to\wT_s\langle-1\rangle 
\end{equation}
such that $\For^\FM_\LM(\zeta_s)=\zeta_s'$, and this morphism satisfies
\begin{equation}
\label{eqn:epsilon-zeta-xi}
\wepsilon_s\langle-1\rangle\circ\zeta_s=\xi.
\end{equation}
And there is a unique isomorphism of graded $R^\vee$-bimodules
\begin{equation}
\label{eq:gamma s}
\gamma_s':R^\vee\otimes_{(R^\vee)^s}R^\vee \langle 1 \rangle \simto 
\HHom_{\FM^{\Kar}(\fh,W)}(\wP,\wT_s)
\end{equation}
sending $u_s=1\otimes 1$ to $\zeta_s$.

The following statement is the analogue in our setting of~\cite[Proposition 3.6]{amrw2}, for which the same proof applies.

\begin{prop}
\label{prop: wP coalgebra}
The object $\wP$ admits a canonical coalgebra structure in the monoidal category $\rT(\fh,W)$, with counit $\xi:\wP\to\wT_1$ and the comultiplication morphism $\delta:\wP\to\wP\wstar\wP$ characterized as the unique morphism such that $(\xi\wstar\xi)\circ\delta=\xi$. \qed
\end{prop}

\subsection{Localization}\label{localization}

We now recall the localization procedure from~\cite[\S11.1]{amrw1}. (Once again, in~\cite{amrw1}, this construction is considered only for Cartan realizations of crystallographic Coxeter groups, but it now makes sense in our present setting.)

Recall the graded ring $Q^\vee$ from~\S\ref{ss:gradings}. If $\cF,\cG$ are objects in $\FM(\fh,W)$ we set
\[
\HHom_{\loc}(\cF,\cG) := \HHom_{\FM(\fh,W)}(\cF,\cG) \otimes_{R^\vee} Q^\vee,
\]
where we consider the \emph{right} action of $R^\vee$ on $\HHom_{\FM(\fh,W)}(\cF,\cG)$ from~\S\ref{ss:FM}.
Let
\[
\Loc(\fh,W)
\]
be the category whose objects are the same as those of $\FM(\fh,W)$, and such that the space of morphisms from $\cF$ to $\cG$ consists of the elements of bi-degree $(0,0)$ in $\HHom_{\loc}(\cF,\cG)$. By construction there exists a canonical functor
\[
\Loc : \FM(\fh,W) \to \Loc(\fh,W).
\]
%
We will also denote by $\Loc'(\fh,W)$ the category obtained from $\FM^{\Kar}(\fh,W)$ by the same procedure as $\Loc(\fh,W)$ is obtained from $\FM(\fh,W)$.
Then there exists a canonical fully faithful functor $\Loc(\fh,W) \to \Loc'(\fh,W)$.

Recall from~\cite[\S7.4]{amrw2} that for any $s \in S$ we have a certain convolutive object $\wnabla_s$ in $\FM(\fh,W)$. Then, for an expression $\uw=(s_1, \cdots, s_n)$ we can consider the object
\[
\wnabla_{\uw}=\wnabla_{s_1}\wstar\cdots\wstar\wnabla_{s_n}.
\]
When $\uw = \varnothing$ we set
$\wnabla_\varnothing:=\wT_\varnothing$.  Using \cite[Lemma 7.4.2 and Corollary 11.3.2]{amrw2}, one sees that we have isomorphisms of graded $Q^\vee$-modules
\begin{equation}
\label{eq:HHom loc costandard}
\HHom_{\loc}(\wnabla_{\ux},\wnabla_{\uy})\cong
\begin{cases}
Q^\vee&\text{if $\pi(\ux)=\pi(\uy)$;}\\                                                      
0&\mbox{otherwise}
\end{cases}
\end{equation}
where $\pi$ is as in~\S\ref{ss:convention}.

Consider the morphism
\begin{equation}
\label{eq:phi s}
\phi_s:\wT_s \to \wnabla_\varnothing\langle1\rangle\oplus\wnabla_s
\end{equation}
in $\FM(\fh,W)$ denoted $\varkappa^2_s$ in~\cite[Proposition 11.2.1]{amrw1}. Given an expression $\uw=(s_1, \cdots, s_n)$, we set
\begin{equation}
\label{eq:phi w}
\phi_\uw:=\phi_{s_1}\wstar\cdots\wstar\phi_{s_n}:\wT_\uw\to(\wnabla_\varnothing\langle1\rangle\oplus\wnabla_{s_1})\wstar\cdots\wstar
(\wnabla_\varnothing\langle1\rangle\oplus\wnabla_{s_n}).
\end{equation}
With this definition, it is clear that for any expressions $\uv,\uw$ we have
\begin{equation}
\label{eqn:phi-product}
\phi_\uv\wstar\phi_\uw=\phi_{\uv\uw}.
\end{equation}

The following statement is our version of~\cite[Corollary~11.2.2]{amrw1}, which follows from the same arguments.

\begin{lema}
\label{le:phis}
 For any expression $\uw$, the morphism $\Loc(\phi_\uw)$ is an isomorphism.\qed
\end{lema}

%

\subsection{Construction of the functor \texorpdfstring{$\V$}{V}}
\label{ss:V}

Recall the functor $(-)^\wedge$ considered in~\S\ref{ss:gradings}. For $\cF \in \TBS(\fh,W)$
we set
\[
\V(\cF):= 
\left( \HHom_{\FM^{\Kar}(\fh,W)}(\wP, \cF) \right)^\wedge.
\]
We observe that $\V\circ\langle1\rangle= \langle -1 \rangle\circ\V$ by definition and~\eqref{eqn:wedge-shift}.
This defines a functor from $\TBS(\fh,W)$ to $R^\wedge\mhyphen\Mod^\Z\mhyphen R^\wedge$.
Our goal in this subsection is to show that this functor factors through a monoidal functor from $\TBS(\fh,W)$ to $\AbeBS(\fh^*,W)$.

If $\uw=(s_1, \cdots, s_n)$ is an expression, we denote by $S_\uw$ the multiset of subexpressions of $\uw$, i.e.~expressions which can be obtained from $\uw$ by removing some letters. (The multiplicities come from the fact that the same expression can sometimes be obtained by removing letters from $\uw$ in several ways.) Then we have
\[
(\wnabla_\varnothing\langle1\rangle\oplus\wnabla_{s_1})\wstar\cdots\wstar
(\wnabla_\varnothing\langle1\rangle\oplus\wnabla_{s_n})\cong \bigoplus_{\uv \in S_{\uw}} \wnabla_{\uv} \langle \ell(\uw)-\ell(\uv)\rangle.
\]
For $u\in W$, we set 
\[
\wT_\uw^u:=\bigoplus_{\substack{\uv \in S_{\uw} \\ \pi(\uv)=u}}\wnabla_{\uv}\langle \ell(\uw)-\ell(\uv)\rangle.
\]
Then $\phi_\uw$ defines a morphism
\[
\wT_{\uw} \to \bigoplus_{u \in W} \wT_{\uw}^u,
\]
and if $\uv$ and $\uw$ are expressions, for $x,y \in W$ we have
\begin{equation}
\label{eq:HHom loc Tu uw}
\HHom_{\loc}(\wT_\uv^x,\wT_\uw^y)=0  \quad \text{unless $x=y$}
\end{equation}
by~\eqref{eq:HHom loc costandard}.

If $\uw$ is an expression and $u \in W$,
we consider the $(R^\wedge,Q^\wedge)$-bimodule
\[
\V(\wT_\uw)_{Q^\wedge}^u := \left( \bigoplus_{n \in \Z} \Hom_{\FM^{\Kar}(\fh,W)}(\wP,\wT_\uw^u \langle -n \rangle) \right)^\wedge \otimes_{R^\wedge} Q^\wedge,
\]
and the morphism of graded $(R^\wedge,Q^\wedge)$-bimodules
\[
\xi_{\V(\wT_{\uw})}:\V(\wT_{\uw})\otimes_{R^\wedge}Q^\wedge \to \bigoplus_{u\in W}\V(\wT_\uw)_{Q^\wedge}^u
\]
induced by the assignment $g\mapsto\phi_\uw\circ g$ for $g \in \Hom_{\rT(\fh,W)}(\wP, \wT_{\uw} \langle -n \rangle)$. It follows from Lemma~\ref{le:phis} that $\xi_{\V(\wT_{\uw})}$ is an isomorphism.

Recall the definition of the category $\cC(\fh^*,W)$ in \S\ref{ss:Abe}. 

\begin{lema}
\phantomsection
\label{lem:V-factorization-Abe}
\begin{enumerate}
 \item 
 \label{it:V-factorization-Abe-1}
 For any expression $\uw$, the triple
 \[
 \left( \V(\wT_\uw),(\V(\wT_\uw)_{Q^\wedge}^u)_{u\in W},\xi_{\V(\wT_{\uw})} \right)
 \]
 is an object of $\cC(\fh^*,W)$.
 \item 
  \label{it:V-factorization-Abe-2}
 For any expressions $\uv,\uw$ and any $\varphi\in\HHom_{\FM(\fh,W)}(\wT_\uv,\wT_\uw)$,
 $\V(\varphi)$ is a morphism in $\cC(\fh^*,W)$ from the triple $(\V(\wT_\uv),(\V(\wT_\uv)_{Q^\wedge}^u)_{u\in W},\xi_{\V(\wT_{\uv})})$ to the triple $(\V(\wT_\uw),(\V(\wT_\uw)_{Q^\wedge}^u)_{u\in W},\xi_{\V(\wT_{\uw})})$.
\end{enumerate}
\end{lema}

\begin{proof}
\eqref{it:V-factorization-Abe-1}
Let $\uw$ be an expression, and let $x\in R^\vee$. For any expression $\uy$ and any $g\in\HHom_{\FM^{\Kar}(\fh,W)}(\wP,\wnabla_{\uy})$ 
we have
\[
x\wstar g=g\circ\mu_{\wnabla_{\uy}}(x)=g\wstar\pi(\uy)^{-1}(x)
\]
by \cite[Lemma 7.4.1]{amrw1}. Therefore the left and right $R^\wedge$-action on $\V(\wT_\uw)_{Q^\wedge}^u$ satisfy the compatibility property in the definition of $\cC(\fh^*,W)$ for all $u\in W$. Hence our triple is an object of $\cC(\fh^*,W)$, as desired.

\eqref{it:V-factorization-Abe-2}
Fix $\uv,\uw$ and $\varphi$ as in the statement. What
we have to prove is that for any $u \in W$ and any $\widetilde{f}\in\V(\wT_\uv)^u_{Q^\wedge}$
the element
\[
h:=\left(\xi_{\V(\wT_{\uw})}\circ(\V(\varphi)\otimes_{R^\wedge}1)\circ(\xi_{\V(\wT_{\uv})})^{-1}\right)(\widetilde{f})
\]
belongs to $\V(\wT_\uw)^u_{Q^\wedge}$.

First, let us assume that $\widetilde{f}=\xi_{\V(\wT_{\uv})}(f)$
for some $f\in\V(\wT_\uv)$. In this case, $h=\phi_\uw\circ\varphi\circ f$ in $\Loc'(\fh,W)$. Since $\phi_\uv$ is an isomorphism in this category, we have
\[
h=(\phi_\uw\circ\varphi\circ\phi_\uv^{-1})\circ(\phi_\uv\circ f)=(\phi_\uw\circ\varphi\circ\phi_\uv^{-1})\circ\xi_{\V(\wT_{\uv})}(f)
\]
and hence $h\in\V(\wT_\uw)^u_{Q^\wedge}$ by our assumption on $f$ and \eqref{eq:HHom loc Tu uw}, see Figure \ref{figure:V varphi}.

\begin{figure}[h]
\begin{tikzpicture}[baseline=(current  bounding  box.center),scale=.8,every node/.style={scale=0.8}]
\node (FM) {
\begin{tikzpicture}
\node (wP) {$\wP$};
\node[right=1cm of wP] (wTv) {$\wT_\uv$};
\node[right=1.5cm of wTv] (wTw) {$\wT_\uw$};
\node[below=1cm of wTv] (wTvu) {$\displaystyle\bigoplus_{u\in W}\wT_\uv^u$};
\node[below=1cm of wTw] (wTwu) {$\displaystyle\bigoplus_{u\in W}\wT_\uw^u$};
\draw[->] (wP) to node[above] {$f$} (wTv);
\draw[->] (wTv) to node[above] {$\varphi$} (wTw);
\draw[->] (wTv) to node[left] {$\phi_\uv$} (wTvu);
\draw[->] (wTw) to node[right] {$\phi_\uw$} (wTwu);
\end{tikzpicture}
};
\node[right=2cm of FM] (Loc) {
\begin{tikzpicture}
\node (wP) {\vphantom{$\wT_\uw$}$\wP$};
\node[rotate=90] at (1.2,-1) {$\circlearrowright$};
\node[right=1.5cm of wP] (wTv) {$\wT_\uv$};
\node[below=1cm of wP] (Tvu) {\vphantom{$\displaystyle\bigoplus_{u\in W}\wT_\uv^u$}$\wT_\uv^u$};
\node[right=2.5cm of wTv] (wTw) {$\wT_\uw$};
\node[rotate=90] at (3.7,-1) {$\circlearrowright$};
\node[below=1cm of wTv] (wTvu) {$\displaystyle\bigoplus_{u\in W}\wT_\uv^u$};
\node[below=1cm of wTw] (wTwu) {$\displaystyle\bigoplus_{u\in W}\wT_\uw^u$};
\draw[->] (wP) to node[above] {$f$} (wTv);
\draw[->,dashed] (wP) to node[left] {$\xi_{\V_{\wT_\uv}}(f)$} (Tvu);
\draw[->,dashed] (wP) to (Tvu);
\draw[right hook->,dashed] (Tvu) to (wTvu);
\draw[->] (wTv) to node[above] {$\varphi$} (wTw);
\draw[->] (wTv) to node[left] {$\phi_\uv$} (wTvu);
\draw[->] (wTw) to node[right] {$\phi_\uw$} (wTwu);
\draw[->,dashed] (wTvu) to node[below] {\tiny$\phi_\uw\circ\varphi\circ\phi_\uv^{-1}$} (wTwu);
\end{tikzpicture}
};
\draw[->,decorate, decoration={snake,pre length=.5cm,post length=.5cm}, shorten <=.5cm, shorten >= .5cm] (FM) to node[above=5pt] {$\Loc'$} (Loc);
\end{tikzpicture}
\caption{The diagram on the left-hand side is in $\FM^{\Kar}(\fh,W)$ and the other one is in $\Loc'(\fh,W)$ where the dashed morphisms exist and their composition equals $h$.}
\label{figure:V varphi}
\end{figure}

Now we consider the general case.
There exist $a\in R^\wedge$ and $f\in\V(\wT_\uv)$ such that $\widetilde{f}=\xi_{\V(\wT_\uv)}(f\frac{1}{a})$. Then $\xi_{\V(\wT_\uv)}(f)\in\V(\wT_\uv)^u_{Q^\wedge}$ since $\xi_{\V(\wT_\uv)}$ is morphism of right $Q^\wedge$-modules, hence $ha \in \V(\wT_\uw)^u_{Q^\wedge}$, which implies that $h$ also belongs to $\V(\wT_\uw)^u_{Q^\wedge}$.
\end{proof}

Lemma~\ref{lem:V-factorization-Abe} shows that the functor $\V$ factors through a functor
\[
\TBS(\fh,W) \to \cC(\fh^*,W),
\]
which will be denoted similarly. We will next show that this functor can be endowed with a monoidal structure. For that,
we introduce the morphism
\begin{equation}
\label{eq:gamma}
\gamma_\varnothing : B^\wedge_\varnothing = R^\wedge \to \V(\wT_\varnothing)=\V(\wT_1) 
\end{equation}
given by the assignment $x\mapsto\xi \wstar x$ (where $\xi$ is as in~\eqref{eq:xi}), and the bifunctorial morphism 
\[
\beta_{\cF,\cG} :\V(\cF)\otimes_{R^\wedge}\V(\cG)\to \V(\cF\wstar \cG)
\]
defined by $\beta_{\cF,\cG}(f\otimes g)=(f\wstar g)\circ\delta\in\V(\cF\wstar\cG)_{m+n}$ for all $f\in\V(\cF)_m$ and $g\in\V(\cG)_n$ (where $\delta$ is as in Proposition~\ref{prop: wP coalgebra}).

\begin{prop}
\label{prop:V is monoidal} 
The triple $(\V,\beta,\gamma_\varnothing)$ is a monoidal functor from $\TBS(\fh,W)$ to $\cC(\fh^*,W)$.
\end{prop}

\begin{proof}
We first observe that the proof of~\cite[Proposition 3.7]{amrw2} shows that $(\V,\beta,\gamma_\varnothing)$ is a monoidal functor from $\TBS(\fh,W)$ to $R^\wedge\mhyphen\Mod^\Z\mhyphen R^\wedge$. In fact, all the ingredients of this proof have been repeated above, except for~\cite[Lemma 2.6]{amrw2}, which can be proved by the same considerations using Proposition~\ref{prop:L1 Delta w nabla w} instead of~\cite[Lemma 4.9]{ar}.

It is clear also that $\gamma_\varnothing$ is a morphism in $\cC(\fh^*,W)$, so all that remains to be justified is that
$\beta_{\cF,\cG}$ is a morphism in $\cC(\fh^*,W)$ for any $\cF,\cG \in \TBS(\fh,W)$, i.e.~that for any expressions $\uv,\uw$ the morphism
\[
\beta_{\uv,\uw}:= \beta_{\wT_{\uv},\wT_{\uw}} : \V(\wT_\uv)\otimes_{R^\wedge}\V(\wT_\uw)\to\V(\wT_{\uv}\wstar\wT_{\uw})
\]
is a morphism in $\cC(\fh^*,W)$. Let $x,y\in W$, $\widetilde{f}\in\V(\wT_\uv)^x_{Q^\wedge}$ and $\widetilde{g}\in\V(\wT_\uw)^y_{Q^\wedge}$. What we have to prove that the element
\[
h:=
\left(\xi_{\V(\wT_{\uv\uw})}\circ(\beta_{\uv,\uw} \otimes_{R^\wedge}1)\circ(\xi_{\V(\wT_{\uv})}^{-1}\otimes_{Q^\wedge}\xi_{\V(\wT_{\uw})}^{-1})\right)(\widetilde{f}\otimes_{Q^\wedge}\widetilde{g})
\]
belongs to $\V(\wT_{\uv\uw})^{xy}_{Q^\wedge}$.
As in the proof of Lemma~\ref{lem:V-factorization-Abe}\eqref{it:V-factorization-Abe-2}, we can assume that there exist $f\in\V(\wT_\uv)$ and $g\in\V(\wT_\uw)$ such that
$\widetilde{f}=\xi_{\V(\wT_{\uv})}(f)\in\V(\wT_\uv)^x_{Q^\wedge}$ 
 and $\widetilde{g}=\xi_{\V(\wT_{\uw})}(g)\in\V(\wT_\uw)^y_{Q^\wedge}$.
 Thus, in $\Loc'(\fh,W)$ we have
\begin{align*}
h=
\xi_{\V(\wT_{\uv\uw})}\left((\beta_{\uv,\uw} \otimes_{R^\wedge}1)((f\otimes_{R^\wedge}1)\otimes_{Q^\wedge}(g\otimes_{R^\wedge}1))\right)&\overset{(a)}{=}\phi_{\uv\uw}\circ (f\wstar g)\circ\delta\\
&\overset{(b)}{=}(\phi_{\uv}\wstar\phi_{\uw})\circ (f\wstar g)\circ\delta\\
&\overset{(c)}{=}(\phi_{\uv}\circ f)\wstar(\phi_{\uw}\circ g)\circ\delta\\
&\overset{(d)}{=}(\xi_{\V(\wT_\uv)}(f)\wstar\xi_{\V(\wT_\uw)}(g))\circ\delta.
\end{align*}
Here:
\begin{itemize}
\item
$(a)$ and $(d)$ follow from the definitions of $\xi_{\V(\wT_{\uv\uw})}$ and $\beta$;
\item
$(b)$ follows from~\eqref{eqn:phi-product};
\item
$(c)$ follows from the interchange law, see~\S\ref{ss:bifunctoriality}.
\end{itemize}
Therefore $h\in\V(\wT_{\uv\uw})^{xy}_{Q^\wedge}$ by our assumptions and \eqref{eq:HHom loc Tu uw}, see Figure \ref{figure:beta in C'}. 
\end{proof}

\begin{figure}[H]
\begin{tikzpicture}[baseline=(current  bounding  box.center),scale=.9,every node/.style={scale=0.9}]
\node (wP) {$\wP$};
\node[right=1cm of wP] (wPwP) {$\wP\wstar\wP$};
\node[right=3cm of wPwP] (wTv) {$\wT_\uv\wstar\wT_\uw$};
\node[below=1cm of wPwP] (wTvwTw) {$\vphantom{\displaystyle\left(\bigoplus_{u\in W}\wT_\uv^u\right)\wstar\left(\bigoplus_{u\in W}\wT_\uv^u\right)}\wT_\uv^x\wstar\wT_\uw^y$};
\node[right=3cm of wTv] (wTw) {$\wT_{\uv\uw}$};
\node[below=1cm of wTv] (wTvu) {$\displaystyle\left(\bigoplus_{u\in W}\wT_\uv^u\right)\wstar\left(\bigoplus_{u\in W}\wT_\uw^u\right)$};
\node[below=1cm of wTw] (wTwu) {$\vphantom{\displaystyle\left(\bigoplus_{u\in W}\wT_\uv^u\right)\wstar\left(\bigoplus_{u\in W}\right)}\displaystyle\bigoplus_{u\in W}\wT_{\uv\uw}^u$};
\draw[->] (wP) to node[above] {$\delta$} (wPwP);
\draw[->] (wPwP) to node[above] {$f\wstar g$} (wTv);
\draw[->] (wTv) to node[left] {$\phi_\uv\wstar\phi_\uw$} (wTvu);
\draw[->] (wTw) to node[right] {$\phi_{\uv\uw}$} (wTwu);
\draw[->,dashed] (wPwP) to node[left] {\small$\xi_{\V(\wT_\uv)}(f)\wstar\xi_{\V(\wT_\uw)}(g)$} (wTvwTw);
\draw[right hook->,dashed] (wTvwTw) to  (wTvu);

\node[rotate=90] at (4,-1) {$\circlearrowright$};

\node[rotate=90] at (8.5,-1) {$\circlearrowright$};

\draw[transform canvas={yshift=-1pt},-] (wTvu) to (wTwu);
\draw[transform canvas={yshift=1pt},-] (wTvu) to (wTwu);

\draw[transform canvas={yshift=-1pt},-] (wTv) to (wTw);
\draw[transform canvas={yshift=1pt},-] (wTv) to  (wTw);
\end{tikzpicture}
\caption{The dashed morphisms are in $\Loc'(\fh,W)$ and the others in $\FM^{\Kar}(\fh,W)$.}
\label{figure:beta in C'}
\end{figure}

Recall the morphism $\gamma_s'$ defined in~\eqref{eq:gamma s}.

\begin{lema}
\label{le:gamma iso}
For any $s \in S$, the morphism $\gamma_s := (\gamma_s')^\wedge$ defines an isomorphism $B_s^\wedge \simto \V(\wT_s)$ in $\cC(\fh^*,W)$.
\end{lema}

\begin{proof}
We already know that $\gamma_s$ is an isomorphism of graded $R^\wedge$-bimodules, so we only have to show that it defines a morphism from $B_s^\wedge$ to $\V(\wT_s)$ in $\cC(\fh^*,W)$. However, the bimodules $(B_s^\wedge)_{Q^\wedge}^u$ and $(\V(\wT_s))^u_{Q^\wedge}$ vanish unless $u \in \{1,s\}$, and identify with the subset of elements $m$ which satisfy $m \cdot x = u(x) \cdot m$ for any $x \in R^\wedge$ if $u \in \{1,s\}$
(see the comments at the beginning of~\cite[\S 2.4]{abe1}), so that the
forgetful functor induces an isomorphism
\[
\End_{\cC(\fh^*,W)}(B^\wedge_s,\V(\wT_s)) \simto \End_{R^\wedge\mhyphen\Mod^\Z\mhyphen R^\wedge}(B^\wedge_s,\V(\wT_s)).
\]
Hence the desired property is automatically satisfied.
\end{proof}

Using Proposition~\ref{prop:V is monoidal} and Lemma~\ref{le:gamma iso} we obtain, for any nonempty expression $\uw$, a canonical isomorphism
\begin{equation}
\label{eq:gamma w}
\gamma_{\uw}:B_\uw^\wedge\simto\V(\wT_\uw) 
\end{equation}
in $\cC(\fh^*,W)$. (For the empty expression, such an isomorphism was already constructed in~\eqref{eq:gamma}.)
%
%
%
%
%

We have therefore proved that the functor $\V$ defines a monoidal functor
\[
\V:\TBS(\fh,W)\to\AbeBS(\fh^*,W) 
\]
such that $\V\circ\langle1\rangle=\langle -1 \rangle\circ\V$, and canonical isomorphisms $\gamma_\uw : B_\uw^\wedge \simto \V(\wT_\uw)$ for all expressions $\uw$.
To finish the proof of Theorem~\ref{thm:V}, it only remains to prove that this functor is fully faithful, which will be done in the next subsection.

\subsection{Invertibility of \texorpdfstring{$\V$}{V}}\label{ss:Invertibility of V}

We start by comparing the graded ranks of morphisms in the source and target categories of $\V$.

\begin{lema}
\label{le:Hom of equal grk}
For any expressions $\uy,\uw$ we have
\[
\grk_{R^\wedge}\HHom_{\TBS(\fh,W)}(\wT_\uy,\wT_\uw)^\wedge=\grk_{R^\wedge}\Hom^\bullet_{\AbeBS(\fh^*,W)}(B^\wedge_\uy, B^\wedge_\uw). 
\]
\end{lema}

\begin{proof}
As in \cite[\S2.7]{amrw2}, we consider the Hecke algebra $\mathcal{H}_{(W,S)}$ of $(W,S)$, its ``standard'' basis $(H_w : w \in W)$ (as a $\Z[v,v^{-1}]$-module) and the $\Z[v,v^{-1}]$-bilinear pairing $\langle -,- \rangle$ on $\mathcal{H}_{(W,S)}$ which satisfies $\langle H_x,H_y\rangle=\delta_{x,y}$, for $x,y\in W$. 
For $s \in S$ we set $\uH_s = H_s + v$, and if $\uw=(s_1, \cdots, s_n)$ is an expression we set
\[
\uH_{\uw} = \uH_{s_1} \cdots \uH_{s_n}.
\]
The expansion of this element in the basis $(H_x : x \in W)$ will be denoted
\[
\uH_\uy=\sum_{x\in W}p_{\uy}^x(v)H_x.
\]

By~\cite[Lemma~2.8]{amrw2} (which holds in our setting, with the same proof), for any expressions $\uy,\uw$ we have
\begin{equation}
\label{eqn:grk-Hom-LM}
\langle \uH_\uy,\uH_\uw\rangle
=\grk_\bk\HHom_{\LM(\fh,W)}(\cT_\uy,\cT_\uw)
=\grk_{R^\vee}\HHom_{\FM(\fh,W)}(\wT_\uy,\wT_\uw)
\end{equation}
where the second equality is due to Proposition \ref{prop: Hom FM between tilting}.
On the other hand, by~\cite[Corollary 4.7]{abe1} we have
\[
\overline{\langle \uH_\uy,\uH_\uw\rangle}=\overline{\sum_{x\in W} p_{\uy}^x(v)p_{\uw}^x(v)}
=\grk_{R^\wedge}\Hom^\bullet_{\AbeBS(\fh^*,W)}(B_\uy^\wedge,B_\uw^\wedge)
\]
This proves the lemma, in view of~\eqref{eqn:grk-wedge}.
\end{proof}

Let us now denote by $R^\vee\mhyphen\Mod^\Z$, resp.~$R^\wedge\mhyphen\Mod^\Z$, the category of graded left $R^\vee$-modules, resp.~$R^\wedge$-modules. Then as for bimodules we have a natural equivalence of categories
\[
(-)^\wedge : R^\vee\mhyphen\Mod^\Z \to R^\wedge\mhyphen\Mod^\Z
\]
which replaces the grading by the opposite grading. We consider the functor
\begin{equation}
\label{eqn:functor-V'}
\V' : \Tilt_\LM(\fh,W) \to R^\wedge\mhyphen\Mod^\Z
\end{equation}
defined by
\[
\V'(\cF) = \HHom_{\LM(\fh,W)}(\cP,\cF)^\wedge.
\]
Then by
Proposition \ref{prop: analogous of Hom FM between tilting} we have a commutative diagram
\[
\xymatrix@C=1.5cm{
\TBS(\fh,W) \ar[r]^-{\V} \ar[d]_-{\For^\FM_\LM} & R^\wedge\mhyphen\Mod^\Z\mhyphen R^\wedge \ar[d]^-{(-)\otimes_{R^\wedge}\bk} \\
\Tilt_\LM(\fh,W) \ar[r]^-{\V'} & R^\wedge\mhyphen\Mod^\Z.
}
\]


The following lemma is an analogue of~\cite[Lemma 3.9]{amrw2}, which follows from the same arguments using Proposition~\ref{prop:L1 Delta w nabla w} instead of~\cite[Lemma 4.9]{ar}.

\begin{lema}
\label{le:V' faithful}
The functor $\V'$ is faithful.
\qed
\end{lema}


We are now ready to complete the proof of Theorem \ref{thm:V}. 

\begin{proof}[Proof of Theorem~\ref{thm:V}] 
As explained at the end of~\S\ref{ss:V}, it only remains to prove that $\V$ is fully faithful.
First we show that $\V$ is faithful. 
Let $\uw$ and $\uv$ be expressions. Fix a finite family of $n$ homogeneous generators of the $R^\vee$-bimodule $\V(\wT_\uv)$; by Proposition~\ref{prop: analogous of Hom FM between tilting} this also provides a family of homogeneous generators of the left $R^\vee$-module $\V'(\cT_\uv)$.
Consider the commutative diagram
\[
\xymatrix@C=0.5cm{
\HHom_{\FM(\fh,W)}(\wT_\uv,\wT_\uw) \ar[r] \ar[d] & \bigoplus_m \Hom_{R^\vee\mhyphen\Mod^\Z\mhyphen R^\vee}(\V(\wT_\uv),\V(\wT_\uw) \langle m \rangle) \ar[r] \ar[d] & \V(\wT_\uw)^{\oplus n} \ar[d] \\
\HHom_{\LM(\fh,W)}(\cT_\uv,\cT_\uw) \ar[r] & \bigoplus_m \Hom_{R^\vee\mhyphen \Mod^\Z}(\V'(\cT_\uv),\V'(\cT_\uw)\langle m \rangle) \ar[r] & \V'(\cT_\uw)^{\oplus n}
}
\]
where:
\begin{itemize}
\item
the left horizontal arrows are induced by the functors $\V$ and $\V'$;
\item
the right horizontal arrows are induced by our choice of generators of $\V(\wT_\uv)$ and $\V'(\cT_\uv)$;
\item
the vertical arrows are induced by the functors $\For_{\LM}^{\FM}$ and $(-) \otimes_{R^\wedge} \bk$.
\end{itemize}
Proposition~\ref{prop: analogous of Hom FM between tilting} shows that in the leftmost column and in the rightmost column, the $\bk$-vector space on the bottom line is obtained from the free right $R^\vee$-module on the upper line by applying the functor $(-) \otimes_{R^\wedge} \bk$.
The lower composition is injective by Lemma \ref{le:V' faithful} and the fact that our family generates $\V'(\cT_\uv)$. By Lemma \ref{le:technical}, it follows that the upper composition is also injective, hence so is the upper left morphism, showing that $\V$ is faithful.

Now, for any $m \in \Z$, the functor $\V$ induces a morphism between the homogeneous components
\[
\left(\HHom_{\FM(\fh,W)}(\wT_\uv,\wT_\uw)^\wedge\right)_{m}\quad\mbox{and}\quad
\Hom_{\AbeBS(\fh^*,W)}(\V(\wT_\uv),\V(\wT_\uw) \langle -m \rangle)
\]
which is injective as $\V$ is faithful. By Lemma \ref{le:Hom of equal grk}, these vector spaces have the same finite dimension. Therefore this morphism is an isomorphism, proving that $\V$ is fully faithful.
\end{proof}

\begin{rmk}
Our proof of full faithfulness above is less direct than that of the corresponding claim in~\cite{amrw2}. This is due to the fact that there is no theory of ``Soergel modules'' in the setting of~\cite{abe1}.
\end{rmk}

\section{Construction of the \texorpdfstring{$2m_{st}$}{2mst}-valent morphisms}
\label{sec:construction}

From now on we drop the assumption that $W$ is finite. We therefore consider an arbitrary Coxeter system $(W,S)$ and a realization $\fh$ of $(W,S)$ which satisfies the conditions of~\S\ref{ss:convention},~\S\ref{ss:dual-real} and~\S\ref{ss:bifunctoriality}.

\subsection{Overview}
\label{ss:overview}

Given a pair $s,t$ of distinct simple reflections such that the product $st$ have finite order $m_{s,t}$, we will denote by $w(s,t)$ the word $(s,t,\cdots)$ of length $m_{s,t}$ where the letters alternate between $s$ and $t$. Note that the relation $\pi(w(s,t)) = \pi(w(t,s))$ in $W$ is precisely the braid relation associated with $s$ and $t$.
Our goal in this section is to construct, for any such pair $s,t$, a canonical morphism
\[
f_{s,t} : \wT_{w(s,t)} \to \wT_{w(t,s)}
\]
which will eventually be the image under Koszul duality of the $2m_{s,t}$-valent morphism associated with $(s,t)$ in the Hecke category. 

We will proceed as follows. By Abe's theory there exists a unique morphism
\[
\varphi_{s,t}:B^\wedge_{w(s,t)}\to B^\wedge_{w(t,s)}
\]
in $\AbeBS(\fh^*,W)$ such that, using the notation of~\S\ref{ss:Abe}, we have
\begin{equation}
\label{eq:varphi st}
\varphi_{s,t}(u_{w(s,t)})=u_{w(t,s)},
\end{equation}
see~\cite[Theorem 3.9]{abe2}. (More specifically, this reference states the existence of such a morphism. Unicity follows from~\cite[Theorem 4.6]{abe1} and the computations in~\cite[\S 4.1]{libedinsky}.) Now, choose a subset $S' \subset S$ which contains $s$ and $t$ and generates a finite subgroup $W'$ of $W$. Choose also the data that allow to define a functor $\V_{W'}$ as in~\S\ref{ss:V} and its monoidal structure, for the Coxeter system $(W',S')$ and its realization $\fh_{|W'}$ (see~\eqref{eqn:wP}--\eqref{eq:xi}). Then the category $\TBS(\fh_{|W'},W')$ identifies in the natural way with a full subcategory of $\TBS(\fh,W)$, in such a way that the objects $\wT_{w(s,t)}$ and $\wT_{w(t,s)}$ in these two categories coincide. By full faithfulness of $\V_{W'}$ (see Theorem~\ref{thm:V}) there exists a unique morphism $f_{s,t}$ as above such that
\begin{equation}
\label{eqn:V-f-phi}
\V_{W'}(f_{s,t})= \gamma_{w(t,s)} \circ \varphi_{s,t} \circ (\gamma_{w(s,t)})^{-1},
\end{equation}
which provides the desired morphism.

Obviously there is a problem with this definition, since it might depend on the extra data we have introduced: the subset $S'$, and the data involved in the definition of $\V_{W'}$. Our goal is exactly to prove that, in fact, it is not the case. Note that we could have solved the problem of the dependence on the choice of $S'$ by setting $S'=\{s,t\}$; however, later we will need to use that~\eqref{eqn:V-f-phi} holds also for some other choices of $S'$.

\subsection{Proof of independence}
\label{ss:coherence of V}



Our goal is therefore to prove the following claim.

\begin{prop}
\label{prop:fst}
For any pair $(s,t)$ of simple reflections generating a finite subgroup of $W$, there exists a morphism
\[
f_{s,t} : \wT_{w(s,t)} \to \wT_{w(t,s)}
\]
which satisfies the following property. For any subset $S' \subset S$ containing $s$ and $t$ and generating a finite subgroup $W'$ of $W$, and for any choices~\eqref{eqn:wP}--\eqref{eq:xi} allowing to define a monoidal functor $\V_{W'}$ as in~\S\ref{ss:V} for the Coxeter system $(W',S')$ and its realization $\fh_{|W'}$, the equality~\eqref{eqn:V-f-phi} holds.
%
\end{prop}

The proof of Proposition~\ref{prop:fst} will use the following preliminary result. Recall, for any $s \in S$, the morphism 
\begin{equation}
\label{eqn:wepsilon-def}
\wepsilon_s:\wT_{s}\to\wT_1\langle1\rangle
\end{equation}
in $\TBS(\fh,W)$ constructed in \cite[\S5.3.4]{amrw1}. For an expression $\uw=(s_1, \cdots, s_r)$ we set
\begin{equation}
\label{eq:wepsilon w}
\wepsilon_{\uw}:=\wepsilon_{s_1} \wstar\cdots \wstar \wepsilon_{s_r}:\wT_{\uw}\to\wT_1\langle \ell(\uw)\rangle.
\end{equation}

\begin{lema}
\label{le:wepsilon}
Let $(s,t)$ be a pair of simple reflections generating a finite subgroup of $W$.
The $\bk$-vector space
\[
\Hom_{\LM(\fh,W)}(\cT_{w(s,t)},\cT_1\langle \ell(w(s,t))\rangle)
\]
is $1$-dimensional, and is spanned by $\For^\FM_\LM(\wepsilon_{w(s,t)})$. Moreover, for any subset $S' \subset S$ generating a finite subgroup $W'$ of $W$, and for any data~\eqref{eqn:wP}--\eqref{eq:xi} allowing to define a monoidal functor $\V_{W'}$ as in~\S\ref{ss:V}, if $f_{s,t}$ is the unique morphism which satisfies~\eqref{eqn:V-f-phi} we have
\[
\For^\FM_\LM(\wepsilon_{w(s,t)})=\For^\FM_\LM(\wepsilon_{w(t,s)}\circ f_{s,t}).
\]
\end{lema}

\begin{proof}
Keeping the notation introduced in the proof of Lemma \ref{le:Hom of equal grk}, by~\eqref{eqn:grk-Hom-LM} we have
\[
p_{w(s,t)}^1(v)=\langle \uH_{w(s,t)},\uH_1\rangle=\sum_{n\in\Z}\dim_\ku\left(\Hom_{\LM(\fh,W)}(\cT_{w(s,t)},\cT_1\langle n\rangle)\right) v^n.
\]
(In Section~\ref{sec:functor V} it is assumed that $W$ is finite, but this condition is not required for this specific statement to hold.)
On the other hand, using the formula in~\cite[Lemma~2.7]{ew} one sees that the highest monomial appearing in $p_{w(s,t)}^1(v)$ is $v^{\ell(w(s,t))}$, and that its coefficient is $1$; we deduce that
\[
\dim_\bk \Hom_{\LM(\fh,W)}(\cT_{w(s,t)},\cT_1\langle \ell(w(s,t))\rangle)=1.
\]
The argument in the proof of~\cite[Lemma 4.4]{amrw2} shows that $\For^\FM_\LM(\wepsilon_{w(s,t)})\neq 0$; this morphism is therefore a generator of $\Hom_{\LM(\fh,W)}(\cT_{w(s,t)},\cT_1\langle \ell(w(s,t))\rangle)$.

Let us now fix data as in the third sentence of the lemma, and consider the corresponding morphism $f_{s,t}$. 
What we have shown above implies that there exists $a \in \bk$ such that
\begin{equation}
\label{eq:a}
\For^\FM_\LM(\wepsilon_{w(t,s)}\circ f_{s,t})=a\cdot\For^\FM_\LM(\wepsilon_{w(s,t)}),
\end{equation}
and what we have to prove is that $a=1$. For this, we will describe some morphisms in different ways and then compare them.

We first compute the morphism
\[
\V_{W'}(\wepsilon_{w(t,s)}\circ f_{s,t}) \left( \gamma_{w(s,t)} (u_{w(s,t)}) \right) \in \V_{W'}(\wT_1)=\Hom_{\FM(\fh,W)}(\wP_{W'},\wT_1),
\]
where $\wP_{W'}$ is the object~\eqref{eqn:wP} we have chosen and, for any $\uw \in \Exp(W')$, $\gamma_{\uw}$ is the isomorphism~\eqref{eq:gamma w} obtained from our choice of morphism~\eqref{eq:xi}.
Explicitly, we note that
\begin{multline*}
\V_{W'}(\wepsilon_{w(t,s)}\circ f_{s,t})\circ\gamma_{w(s,t)}=
\V_{W'}(\wepsilon_{w(t,s)})\circ \V_{W'}(f_{s,t})\circ\gamma_{w(s,t)}\\
=\V_{W'}(\wepsilon_{w(t,s)})\circ\gamma_{w(t,s)}\circ\varphi_{s,t}
\end{multline*}
by~\eqref{eqn:V-f-phi}.
Now we have
\[
(\V_{W'}(\wepsilon_{w(t,s)})\circ\gamma_{w(t,s)}\circ\varphi_{s,t}) (u_{w(s,t)})
= (\wepsilon_{w(t,s)} \langle -\ell(w(t,s)) \rangle ) \circ (\gamma_{w(t,s)}(u_{w(t,s)})).
\]
By definition, the morphism
\[
\gamma_{w(t,s)}(u_{w(t,s)}) : \wP_{W'} \to \wT_{w(t,s)}
\]
is the morphism
\[
(\zeta_t \wstar \zeta_s \wstar\cdots)\circ\delta^{\ell(w(s,t))-1}
\]
where we use the notation from~\eqref{eq:zeta s} and where, for any $n \geq 1$, the morphism
\[
\delta^n : \wP_{W'} \to \underbrace{\wP_{W'} \wstar \cdots \wstar \wP_{W'}}_{\text{$n+1$ times}}
\]
is the $n$-th comultiplication morphism for $\wP_{W'}$, see Proposition~\ref{prop: wP coalgebra}. By definition of $\wepsilon_{w(t,s)}$ (see~\eqref{eq:wepsilon w}), we deduce that
\begin{multline*}
(\V_{W'}(\wepsilon_{w(t,s)})\circ\gamma_{w(t,s)}\circ\varphi_{s,t}) (u_{w(s,t)}) \\
= (\wepsilon_{t}\langle-1\rangle\wstar\wepsilon_s\langle-1\rangle\wstar\cdots) \circ (\zeta_t \wstar \zeta_s \wstar\cdots)\circ\delta^{\ell(w(s,t))-1} \\
= ((\wepsilon_{t} \circ \zeta_t \langle-1\rangle)\wstar (\wepsilon_s \circ \zeta_s \langle-1\rangle) \wstar \cdots) \circ \delta^{\ell(w(s,t))-1}
\end{multline*}
by the exchange law~\eqref{eqn:exchange}. Using~\eqref{eqn:epsilon-zeta-xi} and the axioms of a coalgebra, we finally deduce that
\[
(\V_{W'}(\wepsilon_{w(t,s)})\circ\gamma_{w(t,s)}\circ\varphi_{s,t}) (u_{w(s,t)}) = \xi,
\]
hence that
\[
\left(\V_{W'}(\wepsilon_{w(t,s)}\circ f_{s,t})\circ\gamma_{w(s,t)}\right)(u_{w(s,t)}) = \xi.
\]


Similar considerations show that $\V(\wepsilon_{w(s,t)}) \left( \gamma_{w(s,t)} (u_{w(s,t)}) \right)=\xi$, which finally shows that
\begin{equation}
\label{eq:una igualdad}
\xi = \V_{W'}(\wepsilon_{w(t,s)}\circ f_{s,t}) \left( \gamma_{w(s,t)}(u_{w(s,t)}) \right) = \V(\wepsilon_{w(s,t)}) \left( \gamma_{w(s,t)} (u_{w(s,t)}) \right).
\end{equation}
Using~\eqref{eq:a},
we deduce that
\begin{multline*}
\For^\FM_\LM(\xi) = \For^{\FM}_{\LM} \left( \wepsilon_{w(t,s)}\circ f_{s,t} \circ ( \gamma_{w(s,t)}(u_{w(s,t)}) ) \right) \\
= \For^{\FM}_{\LM} \left( \wepsilon_{w(t,s)}\circ f_{s,t} \right) \circ \For^{\FM}_{\LM} \left(  \gamma_{w(s,t)}(u_{w(s,t)}) \right) \\
= a \cdot \For^{\FM}_{\LM} \left( \wepsilon_{w(s,t)} \right) \circ \For^{\FM}_{\LM} \left(  \gamma_{w(s,t)}(u_{w(s,t)}) \right) \\
= a \cdot \For^{\FM}_{\LM} \left( \wepsilon_{w(s,t)} \circ \gamma_{w(s,t)}(u_{w(s,t)}) \right) = a \cdot \For^\FM_\LM(\xi).
\end{multline*}
Since $\For^\FM_\LM(\xi) \neq 0$ (see the comments following~\eqref{eq:xi}) this implies that $a=1$, which finishes the proof.
\end{proof}

\begin{rmk}
\label{rmk:For-epsilon}
The equalities in~\eqref{eq:una igualdad} and the fact that $\For^\FM_\LM(\xi) \neq 0$ also imply that $\For^\FM_\LM(\wepsilon_{w(s,t)}) \neq 0$.
\end{rmk}

\begin{proof}[Proof of Proposition~\ref{prop:fst}]
What we have to prove is that if $f_{s,t}$ and $f'_{s,t}$ are two morphisms constructed as in Lemma~\ref{le:wepsilon} then $f_{s,t} = f'_{s,t}$. Now the $\bk$-vector space
$\Hom_{\FM(\fh,W)}(\wT_{\ws},\wT_{\wt})$ is  $1$-dimensional, e.g.~because it identifies with the space $\Hom_{\AbeBS(\fh^*,W)}(B^\wedge_{w(s,t)}, B^\wedge_{w(t,s)})$, which is $1$-dimensional as explained in~\S\ref{ss:overview}. This argument also shows that $f_{s,t}$ and $f'_{s,t}$ are nonzero; hence there exists $a \in \bk$ such that $f_{s,t}'=a \cdot f_{s,t}$. Using Lemma~\ref{le:wepsilon} we then obtain that
\[
\For^\FM_\LM(\wepsilon_{w(s,t)})=\For^\FM_\LM(\wepsilon_{w(t,s)}\circ f'_{s,t}) = a \cdot \For^\FM_\LM(\wepsilon_{w(t,s)}\circ f_{s,t}) = a \cdot \For^\FM_\LM(\wepsilon_{w(s,t)}).
\]
In view of Remark~\ref{rmk:For-epsilon} this implies that $a=1$, which finishes the proof.
\end{proof}

\section{Koszul duality}\label{sec:functor Phi}


We continue with the setting of Section~\ref{sec:construction}.

\subsection{Monoidal Koszul duality}
\label{ss:monoidal-Koszul}


The first formulation of our Koszul duality, which generalizes~\cite[Theorem~5.1]{amrw2}, is the following.

\begin{theorem}
\label{thm:Phi}
There exists an equivalence of monoidal categories
\[
\Phi:\DiagBS(\fh^*,W)\simto \TBS(\fh,W)
\]
which satisfies $\Phi\circ(1) = \langle1\rangle\circ\Phi$ and $\Phi(B_\uw^\wedge) = \wT_\uw$ for any $\uw \in \Exp(W)$.
\end{theorem}

The monoidal functor $\Phi$ is constructed in~\S\ref{ss:Phi}--\ref{ss:Phi well defined}. By definition it satisfies $\Phi\circ(1) = \langle1\rangle\circ\Phi$ and $\Phi(B_\uw^\wedge) = \wT_\uw$ for any $\uw \in \Exp(W)$; in particular, it induces a bijection between the sets of objects in $\DiagBS(\fh^*,W)$ and in $\TBS(\fh,W)$. To prove that this functor is an equivalence, it therefore suffices to prove that it is fully faithful, which will be done
in~\S\ref{ss:Phi equivalence}. 

\begin{rmk}
In the case when $W$ is finite, Theorem~\ref{thm:Phi} can be deduced from the combination of Theorem~\ref{thm:V} and Theorem~\ref{thm:Abe-EW} (applied to $\fh^*$). The main point of the proof given below is that it applies also to infinite Coxeter groups.
\end{rmk}

\subsection{Construction of the functor \texorpdfstring{$\Phi$}{Phi}}
\label{ss:Phi}

In view of the definition of the category $\DiagBS(\fh^*,W)$, to define a monoidal functor from $\DiagBS(\fh^*,W)$ to $\TBS(\fh,W)$ which satisfies $\Phi\circ(1) = \langle1\rangle\circ\Phi$ it suffices to define the image of each object $B_s^\wedge$ ($s \in S$), of each generating morphism, and then to check that the images of these morphisms satisfy the defining relations of $\DiagBS(\fh^*,W)$ (in $\TBS(\fh,W)$), see~\S\ref{ss:EW-category}. In this subsection we explain how to define the images of the generating objects and morphisms, and in~\S\ref{ss:Phi well defined} we will show that these images satisfy the required relations.

First, our functor $\Phi$ will send $B_s^\wedge$ to $\wT_s$, for any $s \in S$. By monoidality, for any expression $\uw$, the image of $B^\wedge_\uw$ will then be $\wT_\uw$.


\subsubsection{Polynomials}
\label{sss: Polynomials} 

As we noted in \S\ref{ss:FM}, we have a graded algebra morphism $\mu_{\wT_1}: R^\vee\to\EEnd_{\FM(\fh,W)}(\wT_1)$. Thus, for $x \in R^\wedge_{m}$ we can set
    \[
    \Phi\left(
    \begin{tikzpicture}[thick,scale=0.07,baseline=-0.5ex]
      \node at (0,0) {$x$};
      \draw[dotted] (-5,-5) rectangle (5,5);
    \end{tikzpicture}
    \right):=\mu_{\wT_1}(x):\wT_1\to\wT_1\langle m\rangle.
\]

\subsubsection{Dot morphisms}
\label{sss: Dot morphisms} 

Let $s \in S$.
Recall that we have considered a certain morphism $\wepsilon_s$ in~\eqref{eqn:wepsilon-def}. Consider also the morphism
\[
\weta_s:\wT_1\langle-1\rangle\to\wT_s
\]
defined in~\cite[\S5.3.4]{amrw1}.
We set
\[
\Phi\left(\begin{tikzpicture}[thick,xscale=0.07,yscale=-0.07,baseline=-0.5ex]
      \draw (0,-5) to (0,0);
      \node at (0,0) {$\bullet$};
            \node at (0,-6.7) {\tiny $s$};
    \end{tikzpicture}\right):=\weta_s
    \quad\mbox{and}\quad
    \Phi\left(
    \begin{tikzpicture}[thick,scale=0.07,baseline=-0.5ex]
      \draw (0,-5) to (0,0);
      \node at (0,0) {$\bullet$};
            \node at (0,-6.7) {\tiny $s$};
    \end{tikzpicture} 
    \right):=\wepsilon_s.
\]


\subsubsection{Trivalent vertices}
\label{sss: Trivalent vertices} 

Let again $s \in S$.
We note that the proof of~\cite[Lemma 4.2]{amrw2} is completely diagrammatic, hence that it also applies in our present context. It follows that there exists a unique morphism
\[
\wb_1:\wT_s\to\wT_s\wstar\wT_s\langle-1\rangle, \quad \text{resp.} \quad \wb_2:\wT_s\wstar\wT_s\to\wT_s\langle-1\rangle,
\]
which satisfies
\begin{equation}
\label{eq:trivalent}
(\id_{\wT_s}\wstar\wepsilon_s)\circ\wb_1=\id_{\wT_s},
\quad\text{resp.}\quad
\wb_2\circ(\id_{\wT_s}\wstar\weta_s)=\id_{\wT_s},
\end{equation}
for any $s\in S$. We set
\[
\Phi\left(        \begin{tikzpicture}[thick,baseline=-0.5ex,scale=0.07]
      \draw (-4,5) to (0,0) to (4,5);
      \draw (0,-5) to (0,0);
      \node at (0,-6.7) {\tiny $s$};
      \node at (-4,6.7) {\tiny $s$};
      \node at (4,6.7) {\tiny $s$};      
    \end{tikzpicture}
    \right)=\wb_1
    \quad \text{and} \quad
\Phi\left(
    \begin{tikzpicture}[thick,baseline=-0.5ex,scale=-0.07]
      \draw (-4,5) to (0,0) to (4,5);
      \draw (0,-5) to (0,0);
      \node at (0,-6.7) {\tiny $s$};
      \node at (-4,6.7) {\tiny $s$};
      \node at (4,6.7) {\tiny $s$};    
    \end{tikzpicture}
    \right)=\wb_2.
\]

\subsubsection{\texorpdfstring{$2m_{st}$}{2mst}-valent vertices}
\label{sss: 2mst} 

Let $s,t\in S$ be distinct simple reflections generating a finite subgroup of $W$. We set
%
\[
\Phi\left(            \begin{tikzpicture}[yscale=0.5,xscale=0.3,baseline,thick]
\draw (-2.5,-1) to (0,0) to (-1.5,1);
\draw (-0.5,-1) to (0,0);
\draw[red] (-1.5,-1) to (0,0) to (-2.5,1);
\draw[red] (0,0) to (-0.5,1);
\node at (-2.5,-1.3) {\tiny $s$\vphantom{$t$}};
\node at (-1.5,1.3) {\tiny $s$\vphantom{$t$}};
\node at (-1.5,-1.3) {\tiny $t$};
\node at (-2.5,1.3) {\tiny $t$};
\node at (0.6,-0.7) {\small $\cdots$};
\node at (0.6,0.7) {\small $\cdots$};
\end{tikzpicture}
    \right)=f_{s,t},
\]
where the right-hand side is as in Proposition~\ref{prop:fst}.
    

\subsection{The functor \texorpdfstring{$\Phi$}{Phi} is well defined}
\label{ss:Phi well defined}

To complete the construction of $\Phi$ we now need to check that the morphisms considered in~\S\ref{ss:Phi} satisfy the defining relations of the category $\DiagBS(\fh^*,W)$. Our proof of this fact follows
the same strategy as in \cite[\S4.3]{amrw2}.

Let us fix an arbitrary relation to be checked. There exists a subset $S' \subset S$ generating a finite subgroup $W'$ of $W$ such that this relation involves only words in $S'$. (The subset $S'$ has cardinality at most $3$, but this will not be important for our purposes.) Replacing $(W,S)$ by $(W',S')$, we can therefore assume that $W$ is finite.
Under this assumption we can consider a functor $\V$ as in~Section~\ref{sec:functor V}. Since this functor is fully faithful, the relation under consideration holds in $\TBS(\fh,W)$ if and only if it holds after applying $\V$. 
To check this we will first compute the images under $\V\circ\Phi$ of our generating morphisms, see~\S\ref{sss:relation-poly}--\ref{sss:relations-2m}. 
These computations will show that the isomorphisms $\gamma_\uw$ (see~\eqref{eq:gamma w}) define an isomorphism of monoidal functors between $\V\circ\Phi$ and the functor $\cF:\DiagBS(\fh^*,W)\to\AbeBS(\fh^*,W)$ considered in~\cite[\S3.5]{abe2}. The fact that the relation under consideration holds will then follow from the fact that $\cF$ is indeed well defined, see~\cite[Lemma~3.14]{abe2}.

\subsubsection{Polynomials}
\label{sss:relation-poly}

If $x \in R^\wedge_{m}$, then the morphism
\[
\V\circ\Phi\left(\begin{tikzpicture}[thick,scale=0.05,baseline=-0.5ex]
      \node at (0,0) {$x$};
      \draw[dotted] (-5,-5) rectangle (5,5);
    \end{tikzpicture}\right):\V(\wT_1)\to\V(\wT_1\langle m\rangle)
\]
    is given by
\[    
    f\mapsto\mu_{\wT_1}(x)\circ f=f\circ\mu_{\wP}(x),
 \]
 i.e.~by the action of $x$ on the $R^\wedge$-module $\V(\wT_1) = R^\wedge$.
 

\subsubsection{The upper dot}
\label{sss:image-upper-dot}

By definition
we have that
\[
\V\circ\Phi\left(
    \begin{tikzpicture}[thick,scale=0.07,baseline=-0.5ex]
      \draw (0,-5) to (0,0);
      \node at (0,0) {$\bullet$};
      \node at (0,-6.7) {\tiny $s$};
    \end{tikzpicture} 
    \right)=\V(\wepsilon_s):\V(\wT_s)\to\V(\wT_1\langle1\rangle).
\]
Recall the identification $\gamma_s:B^\wedge_s \simto \V(\wT_s)$ in \eqref{eq:gamma s}, which satisfies $\gamma_s(u_s)=\zeta_s$.
The morphism
\[
\V(\wepsilon_s) \circ \gamma_s : B^\wedge_s \to \V(\wT_1\langle1\rangle) = R^\wedge (1)
\]
is a morphism of $R^\wedge$-modules, which sends the generator $u_s$ to
\[
\wepsilon_s\langle-1\rangle\circ\zeta_s = \xi,
\]
see~\eqref{eqn:epsilon-zeta-xi}, i.e.~to $\gamma_\varnothing(u_\varnothing)$. It therefore coincides with the ``multiplication'' morphism
\[
m_s:B_s^\wedge\to R^\wedge(1)
\]
defined by $a\otimes b\mapsto ab$, see~\cite[\S3.3]{abe1}.

%
%
    
\subsubsection{The lower dot}
\label{sss:image-lower-dot}
    
We now analyze
\[
\V\circ\Phi\left(
    \begin{tikzpicture}[thick,scale=0.07,yscale=-1,baseline=-0.5ex]
      \draw (0,-5) to (0,0);
      \node at (0,0) {$\bullet$};
      \node at (0,-6.7) {\tiny $s$};
    \end{tikzpicture} 
    \right)=\V(\weta_s):\V(\wT_1\langle-1\rangle)\to\V(\wT_s).
    \]
    By \cite[Lemma 5.3.2 $(1)$]{amrw1} we have $\wepsilon_s\circ\weta_s=\mu_{\wT_1}(\alpha_s^\vee)$, and hence
    \[
    \V(\wepsilon_s)\circ\V(\weta_s)=\alpha_s^\vee\cdot\id:\V(\wT_1)\to\V(\wT_1).
    \]
    The morphism $\beta'_s := \gamma_s^{-1} \circ \V(\weta_s) \circ \gamma_\varnothing$ therefore satisfies $m_s\circ\beta'_s=\alpha_s^\vee\cdot\id_{R^\wedge}$.
Now the $\bk$-vector space $\Hom_{\AbeBS(\fh^*,W)}(R^\wedge,B_s^\wedge(1))$ is $1$-dimensional, and generated by the morphism $\beta_s$ such that $\beta_s(1)=\delta_s\otimes 1-1\otimes s(\delta_s)$ where $\delta_s\in V$ is an element which satisfies $\langle\alpha_s,\delta_s\rangle=1$, cf.~\eqref{eq:Bs in Abe}. This generator is the unique vector such that $m_s\circ\beta_s=\alpha_s^\vee\cdot\id_{R^\wedge}$, so that $\beta_s=\beta_s'$.


\subsubsection{Trivalent vertices}

Let $s\in S$. The spaces $\Hom_{\AbeBS(\fh^*,W)}(B^\wedge_s,B^\wedge_{(s,s)}(-1))$ and $\Hom_{\AbeBS(\fh^*,W)}(B^\wedge_{(s,s)},B^\wedge_s(-1))$ are $1$-dimensional by  \cite[Corollary 4.7]{abe1} and an easy computation in the Hecke algebra. As explained e.g.~in~\cite[\S3.5]{abe2}, we can take as generators of these spaces the morphisms
\[
t_1:f\otimes g\mapsto f\otimes 1\otimes g
\quad\mbox{and}\quad
t_2:f\otimes g\otimes h\mapsto f\partial_s(g)\otimes h,
\]
where we use the identification 
\[
B_{(s,s)}^\wedge = B_s^\wedge \otimes_{R^\wedge} B_s^\wedge = R^\wedge \otimes_{(R^\wedge)^s} R^\wedge \otimes_{(R^\wedge)^s} R^\wedge \langle -2 \rangle,
\]
and $\partial_s$ is the Demazure operator associated with $s$, see~\cite[\S 3.3]{ew}

We have
\[
(\id_{B_s^\wedge} \otimes_{} m_s(-1))\circ t_1=\id_{B_s^\wedge} \quad \text{and} \quad t_2(1)\circ(\id_{B_s^\wedge}\otimes_{} \beta_s)=\id_{B_s^\wedge}.
\]
On the other hand, by definition (see~\eqref{eq:trivalent}), we have
\[
(\id_{\V(\wT_s)}\otimes \V(\wepsilon_s)(-1))\circ \V(\wb_1)=\id_{\V(\wT_s)} \quad \text{and}  \quad \V(\wb_2)(1)\circ(\id_{\V(\wT_s)}\otimes\V(\weta_s))=\id_{\V(\wT_s)}.
\]
Using the descriptions in~\S\S\ref{sss:image-upper-dot}--\ref{sss:image-lower-dot}, we deduce that
\[
\V\circ\Phi\left(
    \begin{tikzpicture}[thick,baseline=-0.5ex,scale=-0.07,yscale=-1]
      \draw (-4,5) to (0,0) to (4,5);
      \draw (0,-5) to (0,0);
      \node at (0,-6.7) {\tiny $s$};
      \node at (-4,6.7) {\tiny $s$};
      \node at (4,6.7) {\tiny $s$};    
    \end{tikzpicture}
    \right) = \gamma_{(s,s)} \circ t_1 \circ \gamma_s^{-1}
\]
and
%
%
\[
    \V\circ\Phi\left(
    \begin{tikzpicture}[thick,baseline=-0.5ex,scale=-0.07]
      \draw (-4,5) to (0,0) to (4,5);
      \draw (0,-5) to (0,0);
      \node at (0,-6.7) {\tiny $s$};
      \node at (-4,6.7) {\tiny $s$};
      \node at (4,6.7) {\tiny $s$};    
    \end{tikzpicture}
    \right) = \gamma_{s} \circ t_2 \circ (\gamma_{(s,s)})^{-1}.
    \]
    
    
\subsubsection{\texorpdfstring{$2m_{st}$}{2mst}-valent vertices}
\label{sss:relations-2m}

Let $s,t\in S$ be distinct simple reflections generating a finite subgroup of $W$. By the very definition of $f_{s,t}$ (see Lemma~\ref{prop:fst}) we have
\[
 \V\circ\Phi\left(            \begin{tikzpicture}[yscale=0.5,xscale=0.3,baseline,thick]
\draw (-2.5,-1) to (0,0) to (-1.5,1);
\draw (-0.5,-1) to (0,0);
\draw[red] (-1.5,-1) to (0,0) to (-2.5,1);
\draw[red] (0,0) to (-0.5,1);
\node at (-2.5,-1.3) {\tiny $s$\vphantom{$t$}};
\node at (-1.5,1.3) {\tiny $s$\vphantom{$t$}};
\node at (-1.5,-1.3) {\tiny $t$};
\node at (-2.5,1.3) {\tiny $t$};
\node at (0.6,-0.7) {\small $\cdots$};
\node at (0.6,0.7) {\small $\cdots$};
\end{tikzpicture}
    \right) = \V(f_{s,t})= \gamma_{w(t,s)} \circ \varphi_{s,t} \circ (\gamma_{w(s,t)})^{-1}.
\]

\subsection{Triangulated Koszul duality}
\label{ss:varkappa}


We have now constructed a monoidal functor $\Phi$ as in Theorem~\ref{thm:Phi}. We will still denote by $\Phi$ the induced functor from $\DiagBS^\oplus(\fh^*,W)$ to $\TBS^\oplus(\fh,W)$. In order to prove that this functor is an equivalence, we will first study a variant of this functor obtained by ``killing the right action of $R^\wedge$.'' Namely, recall the constructions of~\S\ref{ss:LE}.
By definition of $\uDiagBS^\oplus(\fh^*,W)$, there exists a unique additive functor
\[
\uPhi:\uDiagBS^\oplus(\fh^*,W)\to\Tilt_\LM^\oplus(\fh,W)
\]
such that
\[
 \For^\FM_\LM\circ\Phi = \uPhi \circ \For^{\BE}_{\LE}.
\]
Recall also the equivalence~\eqref{eq:KTiltLM+-LM}, which we will here denote by $\imath$. Finally, recall that we have an action of the monoidal category $(\TBS(\fh,W),\wstar)$ on the category $\Tilt_{\LM}(\fh,W)$, see Remark~\ref{rmk:bifunctoriality}\eqref{it-action-Tilt}.

%

The following statement is an analogue of~\cite[Theorem 5.2 \& Proposition 5.5]{amrw2}.

\begin{theorem}
\label{thm:varkappa}
The functor 
\[
\varkappa:=\imath \circ \Kb(\uPhi):\LE(\fh^*,W)\to\LM(\fh,W)
\]
is an equivalence of triangulated categories. It satisfies
$\varkappa\circ(1)=\langle1\rangle\circ\varkappa$, and for any $\uv \in \Exp(W)$ and any $w \in W$ we have
\[
\varkappa(\uB_\uv^\wedge) \cong \cT_\uv,\quad\varkappa(\uDelta_w^\wedge)\simeq\dDelta_w,\quad\varkappa(\unabla_w^\wedge)\simeq\dnabla_w.
\]
Finally, for any $\cF\in\DiagBS^\oplus(\fh^*,W)$ and $\cG\in\LE(\fh^*,W)$ we have a canonical isomorphism
\begin{equation}
\label{eqn:varkappa-convolution}
\varkappa(\cF\ustar\cG) \cong \imath \left( \Phi(\cF) \wstar \imath^{-1}(\cG) \right).
\end{equation}
%
\end{theorem}

\begin{proof}
The facts that $\varkappa\circ(1)=\langle1\rangle\circ\varkappa$ and that $\varkappa(\uB_\uv^\wedge) \cong \cT_\uv$ for any expression $\uv$
are true by construction. The isomorphism~\eqref{eqn:varkappa-convolution} is also clear by construction and monoidality of $\Phi$. Next we will prove that $\varkappa(\uDelta_w^\wedge)\simeq\dDelta_w$ for any $w \in W$, following the strategy of \cite[\S 5.2]{amrw2}. The proof that $\varkappa(\unabla_w^\wedge)\simeq\dnabla_w$ is similar, and left to the reader. 

First, for $s\in S$ we consider the functor
\[
C_s':= \imath \circ \Kb(\wT_s\wstar(-)) \circ \imath^{-1} : \LM(\fh,W)\to \LM(\fh,W).
\]
The same construction can be carried out with $\wT_1$ instead of $\wT_s$, providing a functor isomorphic to the identity. These constructions are functorial and we have a morphism of functor $\widetilde{\epsilon}_s:C_s'\to\id\langle1\rangle$ induced by the morphism $\wepsilon_s$, see~\eqref{eqn:wepsilon-def}. As in \cite[Lemma 5.3]{amrw2} one sees that the morphism $\widetilde{\epsilon}_s(\dDelta_w):C_s'(\dDelta_w)\to\dDelta_w\langle1\rangle$ is nonzero for all $w\in W$. On the other hand, recall the functor considered in Lemma~\ref{le:tensoring by wT s is exact}. As in~\cite{amrw2}, there exists an isomorphism of functors
\begin{equation}
\label{eqn:Cs-Ts}
 C_s' \cong \wT_s \wstar (-).
\end{equation}

Now we prove the desired claim by induction on $\ell(w)$. For $w=1$, this claim is true because $\uDelta^\wedge_1=\uB^\wedge_1$ and $\dDelta_1=\cT_1$ by definition.
Let $w\in W$ and $s\in S$ such that $sw<w$, and assume the claim is known for $sw$. By the explicit description of $\Delta_s^\wedge$ in \eqref{eq:Delta s},
we have a distinguished triangle 
\[
\Delta_s^\wedge \to B^\wedge_s\to \Delta^\wedge_{1}(1)\xrightarrow{[1]}
\]
in $\BE(\fh,W)$
where the second morphism is given by the upper dot. Tensoring with $\Delta^\wedge_{sw}$ on the right and using~\cite[Proposition 6.11]{arv}, we obtain a distinguished triangle 
\[
\Delta^\wedge_w \to B^\wedge_s\ustar\Delta^\wedge_{sw}\to\Delta^\wedge_{sw}(1)\xrightarrow{[1]},
\]
where the second morphism is the upper dot tensored by $\id_{\Delta^\wedge_{sw}}$. Applying $\varkappa\circ\For^\BE_{\LE}$ and using~\eqref{eqn:varkappa-convolution} and~\eqref{eqn:Cs-Ts}, we deduce a distinguished triangle 
\[
\varkappa(\uDelta^\wedge_w)\to \wT_s\wstar\varkappa(\uDelta^\wedge_{sw})\to\varkappa(\uDelta^\wedge_{sw})\langle1\rangle\xrightarrow{[1]},
\]
where the second morphism is induced by $\widetilde{\epsilon}_s$ via the identification~\eqref{eqn:Cs-Ts}.
Using the inductive hypothesis, we can rewrite this triangle as
\[
\varkappa(\uDelta^\wedge_w)\to \wT_s\wstar\dDelta_{sw} \xrightarrow{f} \dDelta_{sw}\langle1\rangle\xrightarrow{[1]}
\]
with $f\neq0$.
We compare this triangle with the triangle
\begin{equation}
\label{eq:a triangle}
\dDelta_w\to \wT_s\wstar\dDelta_{sw}\xrightarrow{g} \dDelta_{sw}\langle1\rangle \xrightarrow{[1]}
\end{equation}
provided by Lemma \ref{le:tensoring by wT s}\eqref{it:tensor-Ts-1}. Here we also have $g \neq 0$. We will prove below that
\begin{equation}
\label{eqn:Hom-Ts-Delta}
\dim_\bk\Hom_{\LM(\fh,W)}\left(\wT_s\wstar\dDelta_{sw},\dDelta_{sw}\langle1\rangle\right)=1;
\end{equation}
this will imply that $f$ and $g$ are multiples of each other, hence that $\varkappa(\uDelta_w) \cong \dDelta_w$, which will finish the proof.

In order to prove~\eqref{eqn:Hom-Ts-Delta}, we observe that
\begin{multline*}
\dim_\bk\Hom_{\LM(\fh,W)}\left(\dDelta_{sw}\langle1\rangle,\dDelta_{sw}\langle1\rangle\right)=1
\quad\mbox{and}\\
\Hom_{\LM(\fh,W)}\left(\dDelta_{w},\dDelta_{sw}\langle1\rangle\right)=0=\Hom_{\LM(\fh,W)}\left(\dDelta_{w}[1],\dDelta_{sw}\langle1\rangle\right),
\end{multline*}
where the first two equalities follow from the properties of the standard objects in a highest weight category, and the third one is a consequence of the axioms of a t-structure as the standard objects belong to the heart. Therefore, we obtain~\eqref{eqn:Hom-Ts-Delta}
by applying $\Hom_{\LM(\fh,W)}\left(-,\dDelta_{sw}\langle1\rangle\right)$ to the triangle \eqref{eq:a triangle}. 

Finally, we prove that $\varkappa$ is an equivalence of categories. 
For that, we first notice that, for any $v,w \in W$ and $n,m \in W$, $\varkappa$ induces an isomorphism
\[
\Hom_{\LE(\fh^*,W)}(\uDelta^\wedge_w,\unabla^\wedge_v(n)[m])
\simto
\Hom_{\LM(\fh,W)}(\dDelta_w,\dnabla_v(n)[m]).
\]
In fact these spaces are zero except when $v=w$ and $n=m=0$ in which case they are $1$-dimensional, cf.~\cite[$(9.2)$]{arv}. To prove the claim it therefore suffices to prove that if $g : \uDelta_w^\wedge \to \unabla_v^\wedge$ is a nonzero morphism then $\varkappa(g) \neq 0$. However, it follows from Lemma~\ref{lem:cone-D-N} that the cone of $g$ belongs to the triangulated subcategory of $\LE(\fh^*,W)$ generated by the objects $\uDelta_v^\wedge \langle n \rangle$ with $v \in W$ satisfying $v<w$ and $n \in \Z$, hence the cone of $\varkappa(g)$ belongs to the triangulated subcategory of $\LM(\fh,W)$ generated by the objects $\dDelta_v\langle n \rangle$ with $v \in W$ satisfying $v<w$ and $n \in \Z$. Since $\dnabla_w \oplus \dDelta_w[1]$ does not belong to this subcategory (again, by the recollement formalism), it follows that $\varkappa(g) \neq 0$.

Now, recall that the objects $(\uDelta^\wedge_w \langle n \rangle : w \in W, \, n \in \Z)$ generate the triangulated category $\LE(\fh^*,W)$, and similarly for the objects $(\unabla^\wedge_w \langle n \rangle : w \in W, \, n \in \Z)$, see e.g.~\cite[Lemma~6.9]{arv}. In view of~\cite[Lemma 5.6]{amrw2} (a version of Be{\u\i}linson's lemma), the property checked above therefore implies that $\varkappa$ is fully faithful. Since its essential image contains the objects $(\dDelta_w \langle n \rangle : w \in W, \, n \in \Z)$, which generate $\LM(\fh,W)$ as a triangulated category, this functor is also essentially surjective, which finishes the proof.
%
\end{proof}

\subsection{Full faithfulness of \texorpdfstring{$\Phi$}{Phi}}
\label{ss:Phi equivalence}

As explained in~\S\ref{ss:monoidal-Koszul}, to complete the proof of Theorem \ref{thm:Phi} we only have to prove that $\Phi$ is fully faithful, i.e.~that for any expressions $\uv,\uw$ this functor induces an isomorphism of graded $R^\wedge$-modules
\[
 \bigoplus_{n \in \Z} \Hom_{\DiagBS(\fh^*,W)}(B^\wedge_{\uv},B^\wedge_{\uw}(n)) \simto \bigoplus_{n \in \Z} \Hom_{\TBS(\fh,W)}(\wT_{\uv}, \wT_{\uw} \langle n \rangle).
\]
Now, both sides are graded free of finite rank over $R^\wedge$ by \cite[Corollary~6.14]{ew} and Proposition \ref{prop: Hom FM between tilting}. To prove the claim, using Lemma \ref{le:technical} it therefore suffices to prove that the morphism obtained after applying $(-) \otimes_{R^\wedge} \bk$ is an isomorphism. However, by Proposition \ref{prop: Hom FM between tilting} the latter morphism identifies with the morphism
\[
 \bigoplus_{n \in \Z} \Hom_{\uDiagBS(\fh^*,W)}(\uB^\wedge_{\uv},\uB^\wedge_{\uw}(n)) \to \bigoplus_{n \in \Z} \Hom_{\Tilt_{\LM}(\fh,W)}(\cT_{\uv}, \cT_{\uw} \langle n \rangle)
\]
induced by $\uPhi$; hence it is an isomorphism by Theorem \ref{thm:varkappa}.


\subsection{Self-duality}
\label{ss:self duality}

We now prove an analogue of~\cite[Theorem~5.7]{amrw2}. (This version is called ``self duality'' because the right equivariant and left equivariant categories are canonically equivalent, as explained in~\S\ref{ss:LE}.)
Recall the indecomposable objects $\oB_w \in \oDiagBS(\fh,W)$ and $\uB_w^\wedge \in \uDiagBS(\fh^*,W)$ defined in~\S\ref{ss:EW-category} (for $w \in W$), and the indecomposable objects $\oT_w \in \Tilt_{\RE}(\fh,W)$ and $\uT_w^\wedge \in \Tilt_{\LE}(\fh^*,W)$ defined in~\S\S\ref{ss:tilt}--\ref{ss:LE} (for $w \in W$).


\begin{theorem}
\label{thm:self-duality}
There exists an equivalence of triangulated categories
\[
\kappa:\RE(\fh,W)\simto\LE(\fh^*,W) 
\]
which satisfies $\kappa\circ\langle1\rangle = (1)\circ\kappa$, and such that
\[
\kappa(\oDelta_w)\simeq\uDelta^\wedge_w, \quad
\kappa(\onabla_w)\simeq\unabla^\wedge_w, \quad
\kappa(\oB_w)\simeq \uT_w^\wedge, \quad
\kappa(\oT_w)\simeq \uB^\wedge_w
\]
for any $w\in W$.
\end{theorem}

\begin{proof}
We define $\kappa$ as the inverse of the composition of equivalences $\For^\LM_\RE\circ\varkappa$, see~\eqref{eqn:ForLMRE} and Theorem~\ref{thm:varkappa}. By Theorem~\ref{thm:varkappa} this functor satisfies $\kappa(\oDelta_w)\simeq\uDelta^\wedge_w$ and $\kappa(\onabla_w)\simeq\unabla^\wedge_w$ for any $w \in W$. The fact that $\kappa(\oT_w)\simeq \uB^\wedge_w$ for $w \in W$ is also an immediate consequence of the properties of $\varkappa$. Finally, the fact that $\kappa(\oT_w)\simeq \uB^\wedge_w$ for any $w \in W$ can be deduced exactly as in the proof of~\cite[Theorem~5.7]{amrw2}.
%
%
\end{proof}

\subsection{Application to the combinatorics of indecomposable tilting objects}
\label{ss:application}

There are two families of Laurent polynomials parametrized by pairs of elements of $W$ that one can attach to the realization $\fh$. First, consider the split Grothendieck ring $[\Diag(\fh,W)]_\oplus$ of the monoidal category $\Diag(\fh,W)$, which we consider as a $\Z[v,v^{-1}]$-algebra with $v$ acting via the morphism induced by $(1)$. Recall that, using the notation from the proof of Lemma~\ref{le:Hom of equal grk} (but allowing now $W$ not to be finite), there exists a unique $\Z[v,v^{-1}]$-algebra isomorphism
\[
 \eta : \cH_{(W,S)} \simto [\Diag(\fh,W)]_\oplus
\]
which sends $H_s+v$ to $[B_s]$ for any $s \in S$, see~\cite[Corollary~6.27]{ew}. For $w \in W$ we set
\[
\uH^{\fh}_w = \eta^{-1}([B_w]).
\]
Then $(\uH^{\fh}_w : w \in W)$ is a basis of $\cH_{(W,S)}$, which we call the ``canonical basis'' attached to $\fh$. We can obtain a first family of Laurent polynomials $(h^{\fh}_{y,w} : y,w \in W)$ as the coefficients of the expansion of the elements of this basis in the standard basis $(H_y : y \in W)$; namely for $w \in W$ we have
\[
 \uH^{\fh}_w = \sum_{y \in W} h^{\fh}_{y,w} \cdot H_y.
\]
(The ``$p$-canonical'' basis studied in~\cite{jw} is an example of this construction.) 

On the other hand, consider the objects $(\oT_w : w \in W)$ in $\fP_{\RE}(\fh,W)$. Given $y \in W$ and $n \in \Z$, we will denote by
\[
 (\oT_w : \oDelta_y \langle n \rangle)
\]
the number of times $\oDelta_y \langle n \rangle$ appears in a standard filtration of $\oT_w$. (It is a standard fact that this quantity does not depend on the choice of filtration.) We can then consider, for $y,w \in W$, the Laurent polynomial
\[
 t^{\fh}_{y,w} := \sum_{n \in \Z} (\oT_w : \oDelta_y \langle n \rangle) \cdot v^n.
\]

The following statement shows that these families are exchanged by passing from $\fh$ to $\fh^*$.

\begin{cor}
 For any $y,w \in W$ we have $h^{\fh}_{y,w}=t^{\fh^*}_{y,w}$.
\end{cor}

\begin{proof}
 As explained in~\cite[\S 6.6]{arv}, the natural functor $\DiagBS^\oplus(\fh,W) \to \BE(\fh,W)$ induces an isomorphism
 \[
  [\DiagBS^\oplus(\fh,W)]_\oplus \simto [\BE(\fh,W)]_\Delta
 \]
where the right-hand side is the Grothendieck group of the triangulated category $\BE(\fh,W)$, and the composition of $\eta$ with this isomorphism sends $H_w$ to $[\Delta_w]$ for any $w \in W$. The functor $\For^{\BE}_{\RE}$ also induces an isomorphism
\[
 [\BE(\fh,W)]_\Delta \simto [\RE(\fh,W)]_\Delta,
\]
e.g.~because it is t-exact and induces an isomorphism between the sets of isomorphism classes of simple objects in the heart of the perverse t-structure on both sides. We deduce an isomorphism
\[
 \cH_{(W,S)} \simto [\RE(\fh,W)]_\Delta
\]
sending $H_w$ to $[\oDelta_w]$ for any $w \in W$. In view of the description of morphism spaces between standard and costandard objects (see~\cite[(9.2)]{arv}), the inverse isomorphism sends the class of an object $M$ to
\[
 \sum_{w \in W} \sum_{n,m \in \Z} (-1)^n \dim_\bk \Hom_{\RE(\fh,W)}(M,\onabla_w ( m ) [n]) \cdot v^m H_w.
\]
In particular, for $y,w \in W$ we have
\[
 h^{\fh}_{y,w}(v) = \sum_{n,m \in \Z} (-1)^n \dim_\bk \Hom_{\RE(\fh,W)}(\oB_w,\onabla_y ( m ) [n]) \cdot v^m.
\]

On the other hand, by Theorem~\ref{thm:self-duality}, for any $y,w \in W$ and $n,m \in \Z$ we have
\[
 \Hom_{\RE(\fh,W)}(\oB_w,\onabla_y ( m ) [n]) \cong \Hom_{\LE(\fh^*,W)}(\uT^\wedge_w,\unabla^\wedge_y \langle m \rangle [n]).
\]
Hence this $\bk$-vector space vanishes unless $n=0$, and in this case its dimension is $(\uT^\wedge_w,\uDelta_y \langle m \rangle)$. We deduce that
\[
 h_{y,w}^\fh(v) = \sum_{m \in \Z} (\uT^\wedge_w,\uDelta_y^\wedge \langle m \rangle) \cdot v^m,
\]
which finishes the proof.
\end{proof}

\begin{example}
\label{ex:t-poly-Soergel}
 Consider the realization of Example~\ref{ex:realizations}\eqref{it:Soergel-real}. In this case, the main result of~\cite{ew-hodge} says that for any $y,w \in W$ the polynomial $h^{\fh}_{y,w}$ is up to some factor the Kazhdan--Lusztig polynomial~\cite{kl} attached to $(y,w)$; in the notation of~\cite{soergel} we have $h^{\fh}_{y,w}=h_{y,w}$. 
 Since the dual of a Soergel realization is again a Soergel realization (see Example~\ref{ex:dual-real}\eqref{it:dual-soergel-real}),
 we deduce that for any Soergel realization we have $t^{\fh}_{y,w}=h_{y,w}$.
\end{example}

\begin{rmk}
\label{rmk:koszul-duality}
Continue with the setting of Example~\ref{ex:t-poly-Soergel}. By~\cite[Proposition~8.13]{arv}, for any $w \in W$ we have $\oB_w = \overline{\rL}_w$. As a consequence of Theorem~\ref{thm:self-duality} we therefore have
\begin{equation}
\label{eqn:Ext-simples}
\Hom_{\RE(\fh,W)}(\overline{\rL}_w,\overline{\rL}_w \langle n \rangle [m])=0
\end{equation}
unless $n=-m$. Assume now that $W$ is finite. Then the category $\fP_{\RE}(\fh,W)$ has enough projective objects; if for $w \in W$ we denote by $\overline{\mathscr{P}}_w$ the projective cover of $\overline{\rL}_w$, then setting
\[
\overline{\mathscr{P}} := \bigoplus_{w \in W} \overline{\mathscr{P}}_w, \quad A := \bigoplus_{n \in \Z} \Hom(\overline{\mathscr{P}}, \overline{\mathscr{P}} \langle n \rangle)
\]
we obtain a graded $\bk$-algebra $A$ and an equivalence of categories between $\fP_{\RE}(\fh,W)$ and the category of finite-dimensional graded $A$-modules. Using~\eqref{eqn:Ext-simples} and the techniques of~\cite[\S 9.2]{riche} one can prove that $A$ is a Koszul ring in the sense of~\cite{bgs}. In case $W$ is a finite dihedral group and $\fh$ is the geometric realization, one sees using Theorem~\ref{thm:self-duality} and Ringel duality (see~\cite[Proposition~10.2]{arv}) that $A$ is isomorphic to the graded ring studied in~\cite{sauerwein}.
\end{rmk}


\end{document}